\documentclass{amsart}

\usepackage{amsmath, amssymb, bm, color}
\usepackage{amsthm} 
\usepackage[pagewise]{lineno}
\usepackage{mathrsfs} 

\theoremstyle{definition} 
\newtheorem{Th}{\bf Theorem}[section] 
\newtheorem{Le}{\bf Lemma}
\newtheorem{Co}{\bf Corollary}

\newtheorem{De}{\bf Definition}
\newtheorem{Rem}{\bf Remark}

\newcommand{\N}{\mathbb{N} }

\newcommand{\R}{\mathbb{R} }

\newcommand{\LA}{\left \langle }
\newcommand{\RA}{\right  \rangle }
\newcommand{\LB}{\left \lbrack }
\newcommand{\RB}{\right  \rbrack }
\newcommand{\LC}{\left ( }
\newcommand{\RC}{\right ) }
\newcommand{\LD}{\left \{ }
\newcommand{\RD}{\right \} }
\newcommand{\LZ}{\left | }
\newcommand{\RZ}{\right | }
\newcommand{\DS}{\displaystyle }

\DeclareMathOperator{\sgn}{sgn}

\DeclareMathOperator{\supp}{supp}
\DeclareMathOperator{\id}{id}
\DeclareMathOperator{\dist}{dist}
\DeclareMathOperator{\Lip}{Lip}

\begin{document}
\title[Solvability of Doubly Nonlinear Parabolic Equation]{Solvability of Doubly Nonlinear Parabolic Equation with $p$-Laplacian}
\author[S. Uchida]{Shun UCHIDA}
\address[S. Uchida]{Oita University, Department of Integrated Science and Technology, 
Faculty of Science and Technology, 
700 Dannoharu, Oita City, Oita Pref., 
Japan  870-1192.}
\email{shunuchida@oita-u.ac.jp}
\keywords{Doubly nonlinear equation, 
parabolic type,
initial boundary value problem,
$p$-Laplacian,
well-posedness,
entropy solution}
\subjclass[2010]{Primary 35K92; Secondary 35K61, 47J35, 34G25.}
\thanks{Supported by the Fund for the Promotion of Joint International Research (Fostering Joint International Research (B))
{\#}18KK0073, JSPS Japan.}



\begin{abstract}
In this paper, we consider   
a doubly nonlinear parabolic equation $ \partial _t \beta (u) - \nabla \cdot \alpha (x , \nabla u) \ni f$
with the homogeneous Dirichlet boundary condition in a bounded domain,
where
$\beta : \R \to 2 ^{\R }$ is a maximal monotone graph satisfying $0 \in \beta (0)$
and $ \nabla \cdot \alpha (x , \nabla u )$ stands for a generalized $p$-Laplacian.
Existence of solution to the initial boundary value problem of this equation
has been investigated in an enormous number of papers
for the case where single-valuedness, coerciveness, or some growth condition is imposed on $\beta $.
However, there are a few results 
for the case where such assumptions are removed
and
it is difficult to construct an abstract theory 
which covers the case for $1 < p < 2$.
Main purpose of this paper is to show the solvability of 
the initial boundary value problem for any $ p \in (1,  \infty ) $
without any conditions for $\beta $ except $0 \in \beta (0)$.
We also discuss the uniqueness of solution 
by using properties of entropy solution.

\end{abstract}

\maketitle

\section{Introduction}
In this paper, we are concerned with the initial boundary value problem 
of the following doubly nonlinear equation: 
\begin{equation*}
{\rm (P)}~~
\begin{cases}
~ \partial _t \beta (u (x,t)) - \nabla \cdot  \alpha  (x , \nabla u (x, t))  \ni f(x,t) 
		~&(x,t) \in Q := \Omega \times (0,T), \\
~ u (x,t) = 0 
		~~~&(x,t) \in \partial \Omega \times  (0,T), \\
\end{cases}
\end{equation*}
where $\Omega \subset \R ^d $ ($d \geq 1$) be a bounded domain with a sufficiently 
smooth boundary $\partial \Omega $.
Throughout this paper, we impose the followings  on $\alpha $ and $\beta $:
\begin{itemize}

\item[(H.$\alpha  $)] There exist a  $C^1 $-class function $ a: \Omega \times \R ^d \to  \R $
such that 
$a (x , \cdot ) : \R ^d \to  \R $ is convex
and 
$\alpha  (x , z )= D _z a (x , z)$ holds for every $z  \in \R ^d $ and almost every $x \in \Omega $.
Moreover, $a$ and its derivative
$\alpha  : \Omega \times \R ^d  \to \R ^d $ 
satisfy the followings with some exponent $ p\in (1, \infty )$ and constants $c ,C >0 $:  
\begin{align}
 &c |z| ^p -C  \leq a (x ,z ) \leq C (|z| ^ p +1 ), 
\label{A-a01} \\
&\LZ \alpha (x , z)  \RZ  \leq C \LC |z | ^{p-1} +1 \RC , 
\hspace{3mm} \alpha  (x ,0 ) = 0,
\label{A-a02} 
\end{align}
and 
\begin{equation}
\begin{split}
\LC \alpha (x , z_1) -\alpha (x , z_2)  \RC \cdot (z _1 - z_2) \geq c |z _ 1  - z_ 2| ^p 
\hspace{5mm} &\text{if } p \geq 2, \\
\LC \alpha (x , z_1) -\alpha (x , z_2)  \RC \cdot (z _1 - z_2) 
\geq  \frac{c|z _ 1  - z_ 2| ^2}{|z _ 1 | ^{2-p} + |z_ 2| ^{2-p} +C}  
\hspace{5mm} &\text{if } 1 < p < 2,
\end{split}
\label{A-a03}
\end{equation}
for every $z, z _1 , z_2  \in \R ^n $  and almost every $x \in \Omega $.

\item[(H.$\beta $)] $\beta : \R \to 2 ^{\R }$ is  maximal monotone and  
satisfies $ 0\in \beta (0)$.
\end{itemize}
A typical example of  $\mathcal{A} u (x) := - \nabla \cdot  \alpha  (x , \nabla u (x, t))  $ 
is  the so-called $p$-Laplacian
$- \Delta _p   u :=- \nabla \cdot \LC |\nabla u | ^{p-2} \nabla u \RC $,
which satisfies (H.$\alpha $)
with 
$a (x, z ) = \frac{1}{p} |z| ^p $ and $\alpha (x , z ) = |z| ^{p-2} z$.
The sum of a finite number of $ p_ i $-Laplacian ($i= 1,2, \ldots ,n$)
also fulfills (H.$\alpha $) with $p = \max _{i=1, 2 ,\ldots , n} p _i $.

Putting $\beta (s ) = |s | ^{r -2 } s $ with $r >1 $ in (P), we obtain
\begin{equation*}
\partial _t  (|u| ^{r -2 } u ) - \Delta _p u = f,
\end{equation*}
which has a  huge amount of previous works, e.g.,
\cite{Bam} 
\cite{Ber} 
\cite{JLL} 
\cite{R} 
\cite{T}. 
We also refer to the following equation as an example of (P):
\begin{equation*}
\partial _t v - \Delta \log v = f,
\end{equation*}
which is considered in, for instance,  
\cite{BH} 
\cite{ERV} 
for $d =1 $,
\cite{RVE-01} 
\cite{Vaz-Log} %
\cite{VER-01}
\cite{Wu} 
for $ d= 2 $,
and
\cite{D-dP}
\cite{Qi}
for higher dimensions.
This equation with the boundary condition $\left . v  \right | _{\partial \Omega } \equiv 1$
can be reduce to (P) by $p=2$, $u := \log v $, and $\beta (u) := e ^u -1 $.
In Miyoshi--Tsustumi \cite{MT}, they obtain the following logarithmic diffusion equation 
by passing to the singular limit of a generalized Carleman model:
\begin{equation*}
\partial _t v - \nabla \cdot (|\nabla \log v | ^{p-2} \nabla \log v) = f.
\end{equation*}
Compared with 
above two,
there are  very few investigations for this equation, e.g., in 
\cite{Log-p-Lap}.
In order for (P) to cover such equations possessing strong nonlinearity, 
we need to mitigate the growth condition or coerciveness of $\beta $.
Moreover, if one deal with Stefan problem \cite{Stefan},
Hele-Shaw problem \cite{HeleShaw},
constraint problem ($\beta $ is the subdifferential of an indicator function), or combination of them,
one might face the case where $\beta $ is multi-valued and $D(\beta ) \neq \R $.

The solvability for $p=2$, i.e., $\mathcal{A} = -\Delta $,  
can be derived from the abstract theory of evolution equations
given by B\'{e}nilan \cite{Be} 
and
Br\'{e}zis \cite{Zala}. 
Carrillo \cite{C} 
and 
Kobayashi \cite{KKoba} 
discussed the existence and the uniqueness of solution to (P) with a hyperbolic term. 
As for the case of $p\neq 2$,
Alt--Luckhaus \cite{AL} considered the system of quasilinear doubly nonlinear equations
and assured the solvability of (P) where $\beta $ is single-valued or multi-valued with a growth condition for jump.
We can find the solvability results of (P) in, e.g,
\cite{A-W03} 
\cite{BW} 
\cite{Bl-Fr} 
\cite{C-W99} 
\cite{D-P96} 
for the case where $\beta $ is single-valued
and in
\cite{Ammar} 
\cite{AMTI} 
\cite{B-P} 
\cite{D-E-T} 
\cite{I-U03} 
for multi-valued.
In these study, they imposed some restriction on $\beta $ besides $ 0 \in \beta (0)$
instead of adopting generalized quasilinear terms $\alpha = \alpha  (u , \nabla u)$ or $\alpha  (x , u , \nabla u) $.
In Akagi--Stefanelli \cite{A-S}, they proposed a different approach from those in the above, 
which are based on the time-discretization technique and $L^1$-contraction principle.
They reduce (P) to $0 \in ( \mathcal{A}  ) ^{-1} ( f - \partial _t \xi ) - \beta ^{-1} (\xi ) $,
where $ \xi \in \beta (u)$,
and use the so-called Weighted Energy Dissipation (WED) method.

We next comment on the results for the abstract evolution equation in Banach spaces:
\begin{equation*}
\partial _t Bu  + A u \ni f,
\end{equation*}
where $A$ and $B$ are maximal monotone operators.
Under some  boundedness or coerciveness condition of operators, 
the solvability has been obtained by, e.g., 
\cite{A-H} 
\cite{B-F} 
\cite{D-S} 
\cite{GM} 
\cite{Hokk} 
\cite{K-P} 
\cite{M-M} 
\cite{S-W} 
\cite{Stef} 
\cite{Yama}. 
%
%
On the other hand,
Barbu \cite{Bar} 
removed such 
assumptions 
by using the angular condition between $A$ and the Yosida approximation of $B$
and proved the solvability of (P) with $p \geq 2 $ for any maximal monotone graph satisfying $0 \in \beta (0)$.
Here the condition $p \geq 2 $ seems to be  essential  in this result
since it is hard to obtain the explicit formula of 
the Yosida approximation of $\widetilde {\beta}$, the realization of $ \beta $ in $L^ r (\Omega )$ for $r \neq 2 $.

Main purpose of this paper is to show the existence of solution to (P)
without any restriction to the exponent $p \in (1, \infty )$
and the nonlinearity $\beta $ except $0 \in \beta (0)$.
From viewpoint of  application to specific physical models,
such generalization might seem to be excessive.
However, 
it is still important to complete the solvability result for arbitrary  $\beta $ 
as an auxiliary problem for the classification of $\beta $
by occurrence of extinction phenomena
(see, e.g., \cite{Diaz01}\cite{D-D}).
Another aim of this article is to establish better estimates and regularities of solution,
which may be a useful tool for investigating time-global behavior and perturbation theory.
Although the smoothing property is a fundamental result for standard quasilinear parabolic equations, 
it is not obvious that 
the smoothness of solution is inherited from the given data 
in the doubly nonlinear equation with general nonlinearity
since the monotonicity between the derivatives of $\beta (u) $ and $u $ no longer holds.
Instead of relaxing the requirement of $\beta $, 
we impose some stricter condition on the given initial data, 
external force $f$, and $\alpha $ than those in the previous works given above.
In the next section, 
we fix several notations and state our main results more precisely.
Section 3 and 4 will be expended on the demonstration. 
We first deal with an elliptic equation affiliated with (P) in Section 3
and 
employ the standard time-discretization technique given by  Raviart \cite{R} and Grange--Mignot \cite{GM} in Section 4.
In the final section,
we discuss the uniqueness of solution by following the argument by Carrillo \cite{C},
where the properties of entropy solution (see Kru\v{z}kov \cite{Kru01} \cite{Kru02}) are neatly used. 
Since our arguments  rely on 
those given by \cite{R}, \cite{GM}, and \cite{C},
one may find some duplications in this paper.
However, we attempt to simplify the calculations and give easier proofs in several parts.


\section{Main Theorem}

\subsection{Definition and Notation}

We here collect some notations and basic properties which will be used later
(see e.g., \cite{Attouch} \cite{Bar-1} \cite{Br-1} \cite{Rock}).
 
 %
 %

  
Let $X $ be a Banach space with norm $\| \cdot \| _{X}$ and $X^{\ast} $ be its dual space
with norm $\| \cdot \| _{X^{\ast}}$.
Duality pairing between $X$ and $X^{\ast }$ is denoted by $\LA \cdot , \cdot \RA _{X}$.
Moreover,  $D(A) $ and $R (A) $ stand for the domain and the range of an operator $A$, respectively.
When $A$ is multi-valued, we identify $A$ with its graph $G(A)$ 
and write $[u,v ] \in A $ to describe $ u\in D(A)$ and $v \in A u $.
   An operator $A : X \to 2 ^{X^{\ast}}$  (the power set of $X^{\ast}$)  
  is said to be monotone
  if $\LA u^{\ast} _1 - u^{\ast} _2 , u _1 - u_2 \RC _X \geq 0 $ holds for every 
  $\LB u _i , u^{\ast} _i \RB \in A$ ($i=1,2$)
  and
  a monotone operator $A $ is said to be maximal monotone
  if there is no monotonic extension of $A$. 
  When $X = X^{\ast} $ is a Hilbert space, 
  maximality is  equivalent to $R(\id + \lambda A ) = X $ for any $\lambda >0 $,
  where $\id : X \to X $ stands for the identity mapping.

  A typical example of maximal monotone operator is the subdifferential.
  Let $\phi : X \to ( -\infty , + \infty ]$ be a proper ($\phi \not \equiv + \infty $) 
lower semi-continuous convex functional. We define
the subdifferential operator $\partial _X \phi : X \to 2 ^{X^{\ast}}$ associated with $\phi $ by
  \begin{equation*}
  \partial _X \phi (u ) := \{ u^{\ast} \in X^{\ast} ; \LA u^{\ast} , v- u  \RA _X   
\leq \phi (v) -\phi (u)\hspace{3mm} 
\forall v \in D(\phi )  \} ,
  \end{equation*}
  where $ D(\phi ) := \{ v \in X ;~\phi (v) < +\infty \}$
  is called the effective domain of $\phi $.  
  Obviously $0 \in \partial \phi (u)$ implies that $\phi $ attains its minimum at $u$.
Let $V$ be another Banach space which is densely embedded in $X$ and satisfies $D(\phi ) \subset V$.
Then the restriction of $\phi $ onto $V$
can be regarded as a proper lower semi-continuous convex functional on $V$
and its subdifferential is denoted by  
  \begin{equation*}
  \partial _V \phi (u ) := \{ u^{\ast} \in V^{\ast} ; \LA u^{\ast} , v- u  \RA _V   
\leq \phi (v) -\phi (u)~~~
\forall v \in D(\phi )  \} .
  \end{equation*}
In general, we have $\partial _X \phi \subset \partial _V \phi $,
namely, it holds that
$D(\partial _X \phi ) \subset D(\partial _V \phi)$ and 
$  \partial _X \phi (u) \subset   \partial _V \phi (u)$
for any $u \in D(\partial _X \phi )$.
  We also define
\begin{equation*}
\phi _{\lambda } (u) := \inf _{v \in X } \LD  \frac{\| u-v  \| ^2_X }{2 \lambda } + \phi (v) \RD     
\hspace{5mm} \lambda > 0,
\end{equation*}
which is called the Moreau--Yosida regularization of $\phi $.
If $X$ and $X^{\ast}$ are strictly convex,
$\partial _X \phi _{\lambda }$ coincides with the Yosida approximation of $\partial _X \phi $
with parameter $\lambda >0 $.
Furthermore,  the Legendre--Fenchel transformation (conjugate) of $\phi $
  is defined by 
\begin{equation*}
\phi ^{\ast} (u ^{\ast} ) := \sup _{v \in X } \LD  \LA u ^{\ast} , v \RA _X - \phi (v) \RD     .
\end{equation*} 
 Remark that  $\phi ^{\ast } : X^{\ast} \to (-\infty , + \infty ]$ is proper lower semi-continuous convex
  and $\partial _{X^{\ast}} \phi ^{\ast} = (\partial  _{X} \phi ) ^{-1}$.
  Moreover,
  $ \xi \in  \partial _X \phi (u)$ holds if and only if $ \phi (u) + \phi ^{\ast} (\xi ) = \LA  \xi , u \RA _X$.

For example,
let $X = L ^ r (\Omega )$ with $r \in (1,\infty)$
and define a functional $\varphi $ by
\begin{equation}
\varphi   (u) := 
\begin{cases}
~~\DS  \int _{\Omega }  a (x , \nabla u (x)) dx ~&~ \text{if } u \in L^r (\Omega ) \cap W^{1, p} _0 (\Omega ), \\
~~ +\infty ~&~\text{if }u \in L^r (\Omega ) \setminus W^{1, p} _0 (\Omega ).
\end{cases}
\label{p-Lap} 
\end{equation}
From the assumption (H.$\alpha $), $\varphi $
is a lower semi-continuous and convex functional on $L^{r} (\Omega )$.
Since $\mathscr{D} (\Omega )$ (the set of infinitely differentiable functions with compact support in $\Omega $) 
is dense in $ L^r (\Omega ) \cap W^{1, p} _0 (\Omega )$, 
$\varphi $ is G\^{a}teaux differentiable and $\partial _ {L ^ r  }\varphi 
(u )= \mathcal{A} u $
with domain $
D(\partial _ {L ^ r  }\varphi ) = \{ u\in   L^r (\Omega ) \cap W^{1, p} _0 (\Omega ); \mathcal{A} u \in L^{r'} (\Omega ) \}$,
where
$r ' := r / (r -1 )$ is  the H\"{o}lder conjugate exponent of $r$.
Moreover, by
letting $r = p $ in \eqref{p-Lap}
and restricting $\varphi $ onto $V = W ^{1,p } _0 (\Omega )$,
we can  see that
$\partial _ {W^{1,p} _0  }\varphi 
(u )= \mathcal{A} u $
with domain $
D(\partial _ {W^{1,p} _0   }\varphi ) 
= \{ u\in   W^{1, p} _0 (\Omega ); \mathcal{A} u \in W^{-1, p'}  (\Omega ) \} = W^{1, p} _0 (\Omega )$ by \eqref{A-a02},
where $W^{-1, p'}  (\Omega )$ is the dual of $ W ^{1,p } _0 (\Omega )$.

  If $X = X^{\ast} = \R$, the maximal monotone operator always possesses a primitive function.
  That is to say, if $\beta : \R \to 2 ^{\R}$ is a maximal monotone graph,
  there exist a proper lower semi-continuous convex function $j : \R \to  ( -\infty , + \infty ] $
  such that $\beta = \partial  j $.
  If $0 \in \beta (0)$ holds, we can assume $j (0) = 0 $ and $j \geq 0 $ without loss of generality. 
 Let $X = L^ r (\Omega )$ with $r \in (1, \infty )$, then 
  \begin{equation}
  \psi (u)
  := 
  \begin{cases}
	~~\DS \int _{\Omega } j (u (x)) dx ~~&~\text{ if } u \in L^r (\Omega ) ,~~j (u ) \in L^1 (\Omega ) ,\\
	~~ + \infty ~~&~\text{ otherwise, }
  \end{cases}
  \label{beta} 
  \end{equation}
 is proper lower semi-continuous convex on $L^ r (\Omega )$
 and $\xi \in \partial _{L^ r } \psi  (u)$
  if and only if 
  $\xi \in L^{r'} (\Omega ) $ and $\xi (x) \in \beta (u(x))$ for a.e. $x \in \Omega $.
  In this sense, we write 
 $\widetilde{\beta } :=   \partial _{L^ r } \psi  $ and call it the realization of $\beta $ in $L^r (\Omega )$.
 However, if one consider the subdifferential of the canonical restriction $\psi $ onto $V= W^{1,p}_0 (\Omega )$,
 then $\xi \in \partial _{W^{1,p}_0 } \psi (u)$ dose not
	necessarily imply $ \xi (x) \in \beta (u (x)) $
 except the case of $D(\beta ) = \R $ (see Br\'{e}zis \cite{Br-Conv}).

Define the resolvent of $\beta $ by 
$J_{\lambda }  := (\id + \lambda \beta ) ^{-1 }$
and
the Yosida approximation by
 $\beta _{\lambda } := (\id - J_{\lambda }) /\lambda $.
 Since $\beta _{\lambda }$ is a  Lipschitz continuous function on $\R$,
the realization of $\beta _{\lambda } $ in $L^r (\Omega )$ is Lipschitz continuous if $r \geq 2$.
Note that this does not holds if $r <2 $ in general.
 We also remark that
 the Moreau--Yosida regularization of $\psi $ given in \eqref{beta} 
 dose not coincide with $u \mapsto \int_{\Omega } j_{\lambda } (u (x)) dx $
 except the case of $r=2$,
 where $j _{\lambda } (s) := \inf _{\sigma \in \R} \LD \frac{|s - \sigma |^2}{2 \lambda } + j (\sigma ) \RD$.
 This means the mismatch between the Yosida approximation of $\partial _{L^r} \psi $
 and $\widetilde{\beta _{\lambda }} $, the realization of $\beta _{\lambda }$ in $L^ r (\Omega )$.

 For later use, 
 we here fix some notations for the sign function and the Heaviside function:
 \begin{equation*}
\sgn (s) := \begin{cases}
~~1 ~~&~~\text{ if } s> 0 ,\\
~~[-1 , 1 ] ~~&~~\text{ if } s= 0 ,\\
~~-1  ~~&~~\text{ if } s < 0 ,
\end{cases}
\hspace{5mm} 
H(s) := \begin{cases}
~~1 ~~&~~\text{ if } s> 0 ,\\
~~[0 , 1 ] ~~&~~\text{ if } s= 0 ,\\
~~0  ~~&~~\text{ if } s < 0 .
\end{cases}
\end{equation*}
The Yosida approximations of these maximal monotone graphs
with parameter $\lambda > 0$ coincide with
\begin{equation*}
\sgn _{\lambda } (s) = \begin{cases}
~~1 ~~&~~\text{ if } s> \lambda  ,\\
~~s / \lambda  ~~&~~\text{ if } |s|\leq  \lambda  ,\\
~~-1  ~~&~~\text{ if } s < - \lambda  ,
\end{cases}
\hspace{5mm} 
H_{\lambda } (s) = \begin{cases}
~~1 ~~&~~\text{ if } s> \lambda   ,\\
~~s / \lambda   ~~&~~\text{ if } s\in [ 0 , \lambda ]  ,\\
~~0  ~~&~~\text{ if } s < 0  .
\end{cases}
\end{equation*}
It is well known that
$\beta _{\lambda } (s) $ converges to the minimal section  $\beta ^{\circ } (s)$
for each $s \in D(\beta )$.
The minimal sections of $\sgn $ and $H$ are
\begin{equation*}
\sgn ^{\circ} (s) = \begin{cases}
~~1 ~~&~~\text{ if } s> 0  ,\\
~~0  ~~&~~\text{ if } s= 0  ,\\
~~-1  ~~&~~\text{ if } s < 0  ,
\end{cases}
\hspace{5mm} 
H ^{\circ} (s) = \begin{cases}
~~1 ~~&~~\text{ if } s >   0  ,\\
~~0  ~~&~~\text{ if } s\leq 0  .
\end{cases}
\end{equation*}

\subsection{Main Theorem} 
In this paper,
we  construct a weak solution in the following sense:  
\begin{De}
\label{MDe-01} 
Let $(u_0 , \xi _0) $ satisfy $\xi _0 (x) \in \beta (u_0  (x))$ for a.e. $x \in \Omega $.
 Then $(u , \xi )$ is said  to be a solution to (P)
with initial data $(u_0 , \xi _0) $
if 
it satisfies 
\begin{equation}
\begin{split}
&u \in L^{\infty } (0,T ; W^{1,p } _0 (\Omega )), \\
&\xi \in W^{1 ,\infty } (0,T ; W^{-1,p '}  (\Omega )) \cap L^{\infty } (0,T ; L^{p ' }  (\Omega )), \\
&\xi (x,t ) \in \beta (u (x,t ))~~~\text{ for a.e. } (x,t ) \in Q , \\
&\partial _t \xi (t) + \mathcal{A} u (t) = f(t)  ~~
\text{ in }W^{-1 , p'} (\Omega ) ~~\text{ for a.e. } t  \in (0,T) , \\
& \xi (0) =\xi _0 .
\end{split}
\label{Regu} 
\end{equation}
\end{De}

Then our main results can be stated as follows:
\begin{Th}
\label{MTh-01} 
Assume (H.$\alpha $) and (H.$\beta $).  
Let 
\begin{equation}
f \in W^{1, p' } (0,T ; W^{-1, p'} (\Omega )) \cap L^{\infty} (0,T; L^{p'} (\Omega ) \cap L^{q} (\Omega ))
\label{external} 
\end{equation}
with some $q \in [ 1,  \infty ]$.
Then for any $u _{0} \in W^{1,p} _0 (\Omega)$ and $\xi _{0} \in L^{p'} (\Omega ) \cap L^{q} (\Omega )$
such that $\xi _0(x) \in \beta (u_0 (x))$ for a.e. $x\in \Omega $,
(P) possesses at least one solution
with initial data $(u _ 0 , \xi _0 )$ satisfying
\begin{equation}
\begin{split}
&\sup _{0\leq t \leq T } \|  \xi (t) \| _{L^{p'}} 
		\leq  T \sup _{0\leq t \leq T }   \| f( t)  \| _{L^{p' } }+ \| \xi _0 \| _{L^{ p'}} ,\\
&\sup _{0\leq t \leq T } \| \xi (t) \| _{L^{q}} 
		\leq  T \sup _{0\leq t \leq T }    \| f( t)  \| _{L^{q } }  + \| \xi _0 \| _{L^{ q}} .
\end{split}
\label{MT-02} 
\end{equation}
\end{Th}

\begin{Th}
\label{MTh-02} 
In addition to the assumptions of Theorem \ref{MTh-01},
let $\mathcal{A} u _ 0 \in L^{p'} (\Omega )$.
Then there exist at least one  solution to 
(P) with initial data $(u _ 0 , \xi _0 )$ satisfying \eqref{MT-02} and 
\begin{equation}
\begin{split}
\|  \xi (t) - \xi (s) \| _{L^1} \leq C |t - s | 
\hspace{5mm} \forall t, s \in [0,T],
\end{split}
\label{MT-03} 
\end{equation}
where $C >0 $ is a constant.
\end{Th}
\noindent Existence of solution satisfying \eqref{MT-03}
was obtained in 
Alt--Luckhaus \cite{AL} for Lipschitz continuous $\beta $ and 
Bamberger \cite{Bam} for $\beta (s ) = |s| ^{r -2} r $ with $r >1$
(see also Remark \ref{C-L-theory} below).
We shall show that the time-Lipschitz continuity of solution can be assured  for every maximal monotone graph $\beta $.

In order to state the next result, we here define
\begin{align*}
TV _{\Omega } (\xi ) &:= \sup \LD 
\int_{\Omega } \xi (x) \nabla \cdot \zeta (x) dx;~
 \zeta  = (\zeta _1 , \zeta _2 ,\ldots , \zeta _d ) \in \mathscr{D} (\Omega ),
\| \zeta \| _{L^{\infty }} \leq 1
\RD , \\
BV (\Omega )
& := \{ \xi \in L^1 (\Omega ) ;~ \| \xi \| _{BV} := \| \xi \| _{L^1} + TV _{\Omega }  (\xi ) < \infty \} ,
\end{align*}
and $BV _{0}(\Omega )$ by the closure of $\mathscr{D} (\Omega )$
with respect to the norm $\| \cdot \| _{BV}$
(see \cite{AFP} \cite{Z}).
\begin{Th}
\label{MTh-12} 
In addition to the assumptions of Theorem \ref{MTh-01} (resp.\ Theorem \ref{MTh-02}),
let $\xi _ 0 \in BV  (\Omega )$ and 
$f \in L^{\infty } (0 ,T ; BV _0 (\Omega )) $.
Then there exist at least one  solution to 
(P) with initial data $(u _ 0 , \xi _0 )$ satisfying 
\begin{equation}
\sup _{0\leq t \leq T } TV _{\Omega } (\xi (\cdot , t) ) \leq T \sup _{0\leq t \leq T } TV _{\Omega } (f (\cdot , t) ) 
+ TV _{\Omega } (\xi _0 )
\label{MT-13} 
\end{equation}
in addition to the properties in Theorem \ref{MTh-01} (resp.\ Theorem \ref{MTh-02}).
\end{Th}
\noindent Combining this fact with Theorem \ref{MTh-02},
we can assure the existence of solution belonging to  $BV (Q)$
for general nonlinearity $\beta $
by choosing an appropriate initial value and external force.

Moreover, we can  obtain 
\begin{Th}
\label{MTh-03} 
If $u_0 \in L^{\infty} (\Omega )$ and $q = \infty $ in Theorem \ref{MTh-01} and \ref{MTh-02},
then the solution to (P) satisfies $u \in L^{\infty} (Q)$.
\end{Th}
\noindent Such $L^{\infty}$-estimate of $u$ is already given in  B\'{e}nilan--Wittbold \cite{BW}.
In this paper, we shall give a simple proof
by standard Moser's iteration. 

\begin{Rem}
Regularity given in Definition \ref{MDe-01} is enough to apply Lemma 1.5 of Alt--Luckhaus \cite{AL}.
Namely, it holds that
\begin{equation}
\int_{\Omega } j^{\ast} ( \xi (x, t_2 ))dx  -\int_{\Omega } j^{\ast} ( \xi ( x, t_1 ))dx 
= \int_{t_1}^{t_2}  \LA \partial _t  \xi (t) , u (t ) \RA _{W^{1,p} _ 0 }  dt 
\label{L-AL} 
\end{equation}
 for any $t _ 1 , t_2 \in [0,T]$.
\end{Rem}

\section{Elliptic Problem}
We first deal with the following elliptic problem:
\begin{equation}
\begin{cases}
~~ \xi (x) - \nabla \cdot \alpha (x , \nabla u (x) ) = h (x) ~~~~&~~x \in \Omega ,\\
~~ \xi (x) \in \beta (u (x))  ~~~~&~~x \in \Omega , \\
~~ u(x) = 0   ~~~~&~~x \in \partial \Omega .
\end{cases}
\tag{E$_{h}$}
\label{Ellip} 
\end{equation}
If $D(\beta ) = \R$,
we can easily obtain the solvability
of \eqref{Ellip} for any $h \in W^{-1 , p'} (\Omega )$
by considering the variational problem
of $I (u) := \psi (u) + \varphi (u) - \LA h , u \RA _{W^{1,p}_0}$
on $W ^{1, p} _0 (\Omega)$ and using the result by Br\'{e}zis \cite{Br-Conv},
where 
$\varphi $ and $\psi $
is defined in \eqref{p-Lap} and \eqref{beta} with $r=p$.
Otherwise, however,
$\partial _{W ^{1, p} _0} \psi (u) $ is defined in the sense of distribution 
and we can not  assure that $\xi (x) \in \beta (u (x))$ holds for each $x \in \Omega $.
To cope with this difficulty,
we here slightly restrict the integrability of $h$
and  construct a strong solution to \eqref{Ellip}.
\begin{Th}
\label{Th3-1}
 Assume 
(H.$\alpha $) and  (H.$\beta $) with $p \in (1, \infty )$. 
Let $h \in L^{q } (\Omega) \cap L^{p' } (\Omega )$
with some $q \in [1 , \infty ]$. Then \eqref{Ellip} possesses a unique solution
$u \in W^{1,p} _0 (\Omega ) $
such that $\xi , \mathcal{A} u \in  L^{q } (\Omega) \cap L^{p' } (\Omega )$ and 
\begin{equation}
\| \xi \| _{L ^q} \leq \| h \| _{L^q} ,~~~\| \xi \| _{L ^{p'}} \leq \| h \| _{L^{p'}} ,
\label{Th-E1} 
\end{equation}
where
$\mathcal{A} u (x) := 
- \nabla \cdot (\alpha (x , \nabla u (x)))$.
\end{Th}

\begin{proof}[Proof of Theorem \ref{Th3-1}]
The uniqueness of solution is  shown by the standard energy estimate.
Let $(u _i , \xi _ i) $ ($i =1 ,2 $) be solutions to \eqref{Ellip}.
Testing the difference of equations by $(u_1 - u_2)$,
we  have
\begin{equation*}
\int_{\Omega } (\alpha (x ,\nabla  u_1 (x)) -\alpha (x, \nabla  u_2 (x)) ) (\nabla  u_1 (x) - \nabla   u_2 (x)) dx
\leq 0 .
\end{equation*}
Then we obtain $u_1 \equiv u_ 2$ by  \eqref{A-a03}, and  immediately $\xi _1 \equiv \xi _2$.
Hence we only have to discuss the existence of solution to \eqref{Ellip}.

To this end,
we 
consider the variational problem for  
$I_1 (u) := \psi _1 (u) + \varphi  (u) - \LA h, u \RA _{L^p}$ in $L^{p} (\Omega )$, where
$\varphi $ is given in \eqref{p-Lap} with $r = p $
and 
\begin{equation}
\psi _1 (u) := 
\begin{cases}
~~\DS \int _{\Omega } j _{\lambda } (u (x)) dx ~&~ \text{if }  u \in L^p (\Omega ) ,~ j _{\lambda } (u) \in L^1 (\Omega ), \\
~~ +\infty ~&~\text{otherwise. }
\end{cases}
\label{EL-F-01} 
\end{equation}
Obviously, $\varphi $ and $\psi _1$ are proper lower semi-continuous convex on $L^{p}(\Omega )$.
Remark that $\partial _{L^p} \psi _1 $ may not  coincide with 
the subdifferential of the Moreau--Yosida regularization of \eqref{beta}, i.e.,
the Yosida approximation of $\widetilde{\beta }$
when $p\neq 2$
(for instance, let $\beta = \id$. 
Then the realization of $J_{\lambda } = \beta _{\lambda } = 1 / (1 + \lambda )$
is not Lipschitz continuous on $L^p (\Omega )$ into $L^{p'} (\Omega )$ if $p < 2$).
Since $I _1 $ is bounded from below for any $u \in L^{p} (\Omega )$,
 there exist a global minimizer, denoted by $u _{\lambda }$.

If $p \geq 2$, the Lipschitz continuity of $ \beta _{\lambda   }$ yields 
$D(\widetilde{ \beta _{\lambda }  }) = L^{p} (\Omega )$.
By Theorem 2.10 of \cite{Bar},
we have $\partial _{L^p} I_ 1  = \partial _{L^p} \psi _1 + \partial _{L^p} \varphi  - h  $,
which implies that the minimizer $u _{\lambda }$ becomes 
a unique solution to 
\begin{equation}
\begin{cases}
~~ \beta _{\lambda } (u _{\lambda } (x)) + \mathcal{A} u _{\lambda } (x) = h (x) ~~~~&~~x \in \Omega ,\\
~~ u _{\lambda } (x) = 0   ~~~~&~~x \in \partial \Omega ,
\end{cases}
\label{Ellip-AP} 
\end{equation}
such that 
$u _{\lambda } \in W^{1,p} _0 (\Omega )$ and 
$\widetilde{\beta _{\lambda } }(u _{\lambda } ) ,  \mathcal{A} u_{\lambda } \in L^{p'} (\Omega )$.


Testing \eqref{Ellip-AP} by $u _{\lambda }$, we have
\begin{equation*}
c \| \nabla  u _{\lambda } \| ^{p-1 } _{L^p} \leq C_1 \| h \| _{L^{p'}}    
\end{equation*}
by $\beta _{\lambda } (0) = 0 $, \eqref{A-a02},  and  \eqref{A-a03},
where $C_1  $ is a constant arising from Poincar\'{e}'s inequality.
 Next we multiply \eqref{Ellip-AP} by $k ^{p'}_m ( \beta _{\lambda } (u _{\lambda })) $,
where $m> 0 $ and
\begin{equation*}
k ^{r}_m  (s ) := 
\begin{cases}
~~|s | ^{r -2 } s ~&~\text{if } |s| > m ,\\
~~m ^{r -2 } s ~&~\text{if } |s| \leq m  .
\end{cases}
\end{equation*}
Since $\beta _{\lambda }$ and $k ^{r}_m$  with $r \leq 2$ are Lipschitz continuous
and $k ^{r}_m (\beta _{\lambda } (0))  = 0 $,
we have by \eqref{A-a02}
\begin{equation*}
 \int _{\Omega } k ^{p'}_m ( \beta _{\lambda } (u _{\lambda }(x) ))\mathcal{A}  u _{\lambda } (x)  dx \geq 0 .
\end{equation*}
Moreover, from the fact that
\begin{align*}
&\int_{\Omega} k ^{p ' } _m (\beta _{\lambda } (u _{\lambda }) ) \beta _{\lambda } (u _{\lambda })dx \\
= &
\int_{\{ x \in \Omega ;~ |\beta _{\lambda } (u _{\lambda }(x)) | > m \} }
 | \beta _{\lambda } (u _{\lambda }) | ^{p ' } dx 
 +\int_{\{ x \in \Omega ;~ |\beta _{\lambda } (u _{\lambda }(x)) | \leq m \} }
    m ^{p' -2 } | \beta _{\lambda } (u _{\lambda }) | ^{2 } dx \\
\geq &
\int_{\{ x \in \Omega ;~ |\beta _{\lambda } (u _{\lambda }(x)) | > m \} }
 |k ^{p'}_m ( \beta _{\lambda } (u _{\lambda })) | ^p  dx 
 +\int_{\{ x \in \Omega ;~ |\beta _{\lambda } (u _{\lambda }(x)) | \leq m \} }
    m ^{p' -p } | \beta _{\lambda } (u _{\lambda }) | ^{p } dx \\
		= & \int_{ \Omega }
 |k ^{p'}_m ( \beta _{\lambda } (u _{\lambda })) | ^p  dx ,
\end{align*}
 we  can derive
\begin{equation*}
\LC \int_{ \Omega }
 |k ^{p'}_m ( \beta _{\lambda } (u _{\lambda })) | ^p  dx  \RC ^{1 / p'} 
 \leq \| h \| _{L^{p'}} . 
\end{equation*}
Passing the limit as $m\to 0$ and using  Fatou's lemma, we obtain 
\begin{equation*}
\| \widetilde{ \beta _{\lambda } }(u _{\lambda }) \| _{L^{p'}} \leq  \| h \| _{L^{p'}} .
\end{equation*}
By the same reasoning, it holds that
\begin{equation}
\| \widetilde{ \beta _{\lambda } } (u _{\lambda }) \| _{L^{q}} \leq  \| h \| _{L^{q}} 
\label{Th-E3} 
\end{equation}
for $q \in [1, 2 )$
(substitute  $\sgn _{m } ( \beta _{\lambda } (u _{\lambda }) )$
for  $k ^{q}_m ( \beta _{\lambda } (u _{\lambda }))$ if $q =1 $). 
When $ q \geq  2$, we replace $k ^{q  } _m $ with $K^{q  } _M $,
where $M>0 $ and 
\begin{equation*}
K ^{r}_M (s ) := 
\begin{cases}
~~|s | ^{r -2 } s ~&~\text{if } |s| \leq M ,\\
~~ M ^{r -1 } \sgn s  ~&~\text{if } |s| > M  .
\end{cases}
\end{equation*} 
Since $K ^{r}_M$ is Lipschitz continuous  for $r \geq 2$ and satisfies 
$K ^{r}_M (0) = 0$, 
 we have by \eqref{A-a02} 
\begin{equation*}
 \int _{\Omega } K ^{q}_M ( \beta _{\lambda } (u _{\lambda } (x)))
 \mathcal{A}  u _{\lambda } (x)  dx \geq 0 
\end{equation*}
and
\begin{align*}
&\int_{\Omega} K ^{q  } _M (\beta _{\lambda } (u _{\lambda }) ) \beta _{\lambda } (u _{\lambda })dx \\
= &
\int_{\{ x \in \Omega ;~ |\beta _{\lambda } (u _{\lambda }(x)) | \leq M \} }
 | \beta _{\lambda } (u _{\lambda }) | ^{q  } dx 
 +\int_{\{ x \in \Omega ;~ |\beta _{\lambda } (u _{\lambda }(x)) | > M \} }
    M ^{q -1 } | \beta _{\lambda } (u _{\lambda }) |  dx \\
\geq &
\int_{\{ x \in \Omega ;~ |\beta _{\lambda } (u _{\lambda }(x)) | \leq M \} }
 | \beta _{\lambda } (u _{\lambda }) | ^{q  } dx 
 +\int_{\{ x \in \Omega ;~ |\beta _{\lambda } (u _{\lambda }(x)) | > M \} }
    M ^{q  }  dx \\
		= & \int_{ \Omega }
 |K ^{q}_M ( \beta _{\lambda } (u _{\lambda })) | ^{q'}  dx ,
\end{align*}
which implies \eqref{Th-E3} with $q \in [2 , \infty )$.


By uniform boundedness given above, there exist a subsequence $\{ \lambda _n \} _{n \in \N }$
of $\{ \lambda \}$
such that $\lambda _n \to 0$ and
\begin{equation*}
\begin{cases}
~~u _{\lambda _n} \to \exists u ~~&~\text{strongly in }L^p(\Omega )\text{ and weakly in }W^{1,p} _ 0 (\Omega ), \\
~~\widetilde{\beta _{\lambda _n} }  (u _{\lambda }) \to \exists \xi  ~~&~\text{weakly in }L^q(\Omega )\text{ and }L^{p'}(\Omega ),\\
~~\mathcal{A} u _{\lambda _n} \to \exists \eta  ~~&~\text{weakly in }L^q(\Omega )\text{ and }L^{p'}(\Omega ),
\end{cases}
\end{equation*}
as $n \to \infty $ (we use the Dunford-Pettis theorem for $q = 1$).
The demi-closedness of maximal monotone operator 
leads to $\eta =  \mathcal{A} u$.
Since 
$\widetilde{J _{\lambda _n} }u _{\lambda _n}  - \widetilde{J_{\lambda _n } }u \to 0$ strongly in $L^{p} (\Omega )  $ and 
$\widetilde{J _{\lambda _n} } u _{\lambda _n}  -u _{\lambda _n} \to 0 $ strongly in $L^{p' } (\Omega )  $ by 
the Lipschitz continuity of $J _{\lambda } $ and uniform boundedness of $
\| \widetilde{\beta _{\lambda }} (u_{\lambda }) \| _{L^{p'}}$,
we can see that $\widetilde{J _{\lambda _n} } u \to u $
 strongly in $L^{p'} (\Omega )  $ (remark that $p' \leq p $).
Then from 
Lebesgue's dominated convergence theorem,
it follows that 
$\widetilde{J _{\lambda _n} }u _{\lambda _n} \to  u $ strongly in $L^{p} (\Omega )  $.
Hence by the maximal monotonicity of $\widetilde{\beta }$ and the fact that
$\beta _{\lambda } (u _{\lambda } (x)) \in \beta (J_{\lambda } u_{\lambda } (x) ) $ for a.e.\ $x \in \Omega $,
we can see  that  $\xi \in \widetilde{\beta }(u )$ in $L^{p'} (\Omega) $,
i.e.,   $\xi (x) \in \beta (u (x))$ for a.e.\  $x \in \Omega$.
%
Therefore the equation of \eqref{Ellip-AP} weakly converges to the equation of \eqref{Ellip} in $L^{p'} (\Omega )$
and the limit $u$ is a (unique) solution to  \eqref{Ellip}.
Moreover, \eqref{Th-E1} can be obtained as  the limit of \eqref{Th-E3} as $\lambda \to 0 $.

Next  suppose that $1 < p < 2$,
where we can not confirm 
whether $D(\partial _{L^p }\psi _1) $ coincides with  $L^{p } (\Omega )$.
To cope with this difficulty,
 we begin with 
the variational problem of 
$I _{2} (u): = \frac{\varepsilon }{2} \| u \| ^2 _{L^2} + \psi _1 (u) + \varphi (u) - \LA h , u \RA _{L^2}$
(recall $h \in L^{p'} (\Omega ) \cap L^{q} (\Omega ) \subset L^{2} (\Omega )$),
where $\varepsilon >0 $,
$\varphi $ is given in \eqref{p-Lap} with $r = 2 $,
and
$\psi _1 $ in \eqref{EL-F-01} with $L^p$ replaced by $L^2$.
Thanks to the approximation term $\frac{\varepsilon }{2} \| u \| ^2 _{L^2}$, 
$I _{2} $ attains its minimum in $L^2 (\Omega )$.
Since 
$D (\partial _{L^2 }\psi _1  ) = L^2 (\Omega )$,
the global minimizer $u_{\varepsilon }$ is a (unique) solution to the following Euler-Lagrange 
equation:
\begin{equation}
\begin{cases}
~~ \varepsilon u_{\varepsilon } +
		\beta _{\lambda } (u_{\varepsilon }(x)) + \mathcal{A} u_{\varepsilon } (x) = h (x) ~~~~&~~x \in \Omega ,\\
~~ u_{\varepsilon } (x) = 0   ~~~~&~~x \in \partial \Omega ,
\end{cases}
\label{Ellip-AP2} 
\end{equation}
where 
$u_{\varepsilon } \in D(\partial _{L^2 }\psi _1) \cap D(\partial _{L^2 }\varphi)$,
i.e.,
$u_{\varepsilon } \in L^2 (\Omega ) \cap W^{1, p} _0 (\Omega )$
and $\widetilde{\beta _{\lambda }} (u_{\varepsilon } ) , \mathcal{A} u_{\varepsilon } \in L^2(\Omega )$.

Multiplying \eqref{Ellip-AP2} by $u _{\varepsilon }$, we get 
 \begin{equation*}
\frac{\varepsilon }{2} \| u _{\varepsilon } \| ^2 _{L^2}  +c \| \nabla u _{\varepsilon } \| ^p _{L^p} 
\leq C_2 \| h \| ^{p'} _{L^{p'}},
\end{equation*}
where $C_2 > 0$ is a  general constant independent of $\varepsilon $.
Testing \eqref{Ellip-AP2} by $K ^{p'} _M  (u _{\varepsilon } )$,
we have
 \begin{equation*}
\varepsilon  \| K ^{p'} _M (u _{\varepsilon } ) \| ^p _{L^p} 
\leq  \| h \|  _{L^{p'}}  \| K ^{p'} _M (u _{\varepsilon } ) \|  _{L^p} ,
\end{equation*}
which implies 
 \begin{equation*}
\varepsilon  \| u _{\varepsilon }  \|  _{L^{p'}} 
=
\varepsilon  \| |u _{\varepsilon }|^{p' -2}u _{\varepsilon }   \|  ^{p-1}_{L^{p}} 
\leq  \| h \|  _{L^{p'}}.
\end{equation*}
Repeating the same procedures as that for $p \geq 2 $, 
we obtain 
 \begin{equation*}
 \| \beta _{\lambda } (u _{\varepsilon } ) \|  _{L^{p'}} 
\leq  \| h \|  _{L^{p'}},
~~~
 \| \beta _{\lambda } (u _{\varepsilon } ) \|  _{L^{q}} 
\leq  \| h \|  _{L^{q}}.
\end{equation*}
Returning to \eqref{Ellip-AP2},
we get $\| \mathcal{A} u_{\varepsilon } \| _{L^{p'}} \leq 3 \| h \| _{L^{p'}}$.
These estimates imply that
$u_{\varepsilon }$ belongs to 
$D(\partial _{L^p} \varphi) \cap D(\partial _{L^p} \psi _1  )$,
which is included in 
$D(\partial _{L^2} \varphi) \cap D(\partial _{L^2} \psi _1  )$.
Hence we can extract a subsequence of $\{ \varepsilon \}$ (we omit relabeling)
such  that  
\begin{equation*}
\begin{cases}
~~\varepsilon  u_{\varepsilon  } \to 0 ~~&~\text{strongly in }L^2(\Omega )\text{ and weakly in }L^{p'}  (\Omega ),\\
~~u _{\varepsilon } \to \exists u_{\lambda } ~~&~\text{strongly in }L^p(\Omega )\text{ and weakly in }W^{1,p} _ 0 (\Omega ), \\
~~\beta _{\lambda }  (u _{\varepsilon }) \to \exists \xi _{\lambda } 
			~~&~\text{weakly in }L^q(\Omega )\text{ and }L^{p'}(\Omega ),\\
~~\mathcal{A} u _{\varepsilon } \to \exists \eta _{\lambda } ~~&~\text{weakly in }L^q(\Omega )\text{ and }L^{p'}(\Omega ).
\end{cases}
\end{equation*}
The maximal monotonicity of $\partial _{L^p} \varphi _1 $ and $\partial _{L^p} \psi _1$ 
yield 
$\eta _{\lambda }  = \mathcal{A} u_{\lambda }$
and $\xi _{\lambda }  = \widetilde{\beta _{\lambda } }(u _{\lambda })$.
Therefore the limit $u _{\lambda }\in W^{1,p}_0 (\Omega ) $ is 
a (unique) solution to \eqref{Ellip-AP} for $1 < p <2 $.

By tracing  the same procedures as for $p \geq 2$, we obtain
\begin{equation*}
\| \nabla  u_{\lambda } \| ^{p-1}_{L^{p}} \leq C_2 \| h \| _{L^{p'}},~~
\| \widetilde{\beta _{\lambda } }(u _{\lambda }) \| _{L^{p'}} \leq \| h \| _{L^{p'}},~~
\| \widetilde{\beta _{\lambda } }(u _{\lambda }) \| _{L^{q}} \leq \| h \| _{L^{q}},
\end{equation*}
and 
\begin{equation*}
\begin{cases}
~~u _{\lambda } \to \exists u ~~&~\text{strongly in }L^p(\Omega )\text{ and weakly in }W^{1,p} _ 0 (\Omega ), \\
~~\widetilde{\beta _{\lambda } } (u _{\lambda }) \to \exists \xi  ~~&~\text{weakly in }L^q(\Omega )\text{ and }L^{p'}(\Omega ),\\
~~\mathcal{A} u _{\lambda } \to \mathcal{A} u ~~&~\text{weakly in }L^q(\Omega )\text{ and }L^{p'}(\Omega ),
\end{cases}
\end{equation*}
where $ u_{\lambda }$ is a solution to \eqref{Ellip-AP} with $1 < p <2 $.
Furthermore, 
$\widetilde{ J_{\lambda  } }u_{\lambda } -\widetilde{ J_{\lambda  }}u \to 0  $ in $L^p (\Omega )$
and 
$\widetilde{J_{\lambda  } }u_{\lambda } - u_{\lambda  } \to 0  $ in $L^{p'} (\Omega )$
directly lead to $\widetilde{J_{\lambda  }}u_{\lambda } \to u $ strongly in $L^p (\Omega )$
since $p < p'$.
Therefore
we can assure that  $\xi \in \widetilde{\beta }  (u) $
and the limit $u$ is a required solution to \eqref{Ellip}.

If  $h \in L^{\infty } (\Omega )$,
we obtain  \eqref{Th-E1} with $q = \infty $
by taking the limit  as $q \to \infty $
(see Ch.1 {\S}3 Theorem 1 of \cite{Y}).
\end{proof}



\begin{Rem}
\label{Estimate-Ellip} 
For later use, we here establish several a priori estimates of $u $ and $\xi $
with some restrictions on $h$, $\alpha $ and $\beta $.

 1. If $h \in L^{\infty } (\Omega )$, we obtain $u \in L ^{\infty } (\Omega ) $
 by  Moser's iteration technique.
Indeed,  testing \eqref{Ellip} by $K^{p(r - 1) +2 } _M( u)$ with $r \geq 1$
and using \eqref{A-a02},
we have 
 \begin{align*}
  |\Omega |  ^{1 - \frac{p(r-1) +1}{pr}} \| h \| _{L^{\infty }}  \| u  \| _{L^{pr}}  ^{p(r-1) +1 } 
 \geq  &  c r^{-p}  \lim _{M\to \infty }
		 \|  K ^{r+1} _M ( u ) \| ^p _{L^{p\mu } }   \\
 \geq &   c r^{-p}  \|  u  \| ^{pr} _{L^{pr \mu } } .
 \end{align*}
 where $|\Omega |$ is the measure of $\Omega $ and 
 $\mu > 1 $ is a exponent arising from  Sobolev's inequality. 
 Hence putting $r := \mu ^ l $ with $l= 0 , 1 , \ldots $
and letting $l \to \infty $,
we obtain 
$\| u \| _{L^{\infty }} \leq C $,
where $C$ is depending only on $\| h  \| _{L^{\infty }}$, $\mu $, and $|\Omega |$.

2. If $\alpha (x, z) = \alpha (z)$ and  $h \in \mathscr{D} (\Omega )$,
 we obtain $\xi \in BV (\Omega )$ (see Theorem 6 of \cite{C}).
We first prove it for the case where $\beta $ is Lipschitz continuous.
Define 
\begin{equation*}
\beta ^{\varepsilon } (s) 
:= \begin{cases}
~~ \beta (s) - \beta (\varepsilon ) ~~&~~\text{ if } s \geq \varepsilon , \\
~~ 0 ~~&~~\text{ if } - \varepsilon \leq s \leq \varepsilon , \\
~~ \beta (s) - \beta ( - \varepsilon ) ~~&~~\text{ if } s \leq  - \varepsilon ,
\end{cases}
\hspace{5mm} \varepsilon >0 . 
\end{equation*}
Clearly, 
$\beta ^{\varepsilon } $ is Lipschitz continuous with the same constant as $\beta $
and converges to $ \beta $ uniformly in $\R$.
Let $u ^{\varepsilon } \in L^{\infty} (\Omega ) \cap W^{1 ,p } (\Omega )$
be a solution to 
$\beta ^{\varepsilon } (u ^{\varepsilon }) - \mathcal{A} u ^{\varepsilon } = h $.
Since $h - \beta ^{\varepsilon } ( u ^{\varepsilon }) \in L^{\infty } (\Omega )$,
$u ^{\varepsilon }  $ is continuous on $\overline{\Omega}$
(see Ch.\ 4  Theorem 1.1 of \cite{LU}).
By homogeneous Dirichlet boundary condition,
the set $\omega ^{\varepsilon } := \{  x \in \Omega ; |u ^{\varepsilon } (x)| \geq \varepsilon \}$
is compact and included in $\Omega $.
Since $\beta ^{\varepsilon  } ( u ^{\varepsilon  }  ) = 0 $ in $\Omega \setminus \omega ^{\varepsilon } $,
it holds that  $\supp \mathcal{A} u^{\varepsilon } \subset 
( \supp \beta ^{\varepsilon  } ( u ^{\varepsilon  }  ) \cup  \supp h  )  = : \Omega ^{\varepsilon } 
\subset \subset \Omega $.
Then
for any $e \in \R ^d $ such that  $|e| < \dist ( \Omega ^{\varepsilon } , \Omega ^c ) $,
we can verify the following standard calculation of translation method
(use here \eqref{A-a03} and $\alpha (x, z) = \alpha (z)$):
\begin{align*}
0 &\leq \lim_{\lambda \to 0} 
\int _{\Omega } (\alpha (\nabla u ^{\varepsilon  }  (x + e)) -\alpha (\nabla u ^{\varepsilon  }  (x )) ) 
\cdot 
\nabla \sgn _{\lambda } (u ^{\varepsilon  }  (x + e) - u^{\varepsilon  }   (x )) dx\\
&=  
- 
\int _{\Omega } | \beta ^{\varepsilon  } (u ^{\varepsilon  }  (x + e) )- 
\beta ^{\varepsilon  } (u ^{\varepsilon  }  (x ) ) | dx
+ 
\int _{\Omega } | h  (x + e) - h    (x ) | dx ,
\end{align*}
which implies
$TV _{\Omega } ( \beta ^{\varepsilon  } (u ^{\varepsilon  }) ) 
\leq  TV _{\Omega } ( h ) $.
Repeating a priori estimates above,
we have the uniform boundedness of $\| u ^{\varepsilon  }\| _{W^{1,p}}$,
$  \| \mathcal{A} u ^{\varepsilon  } \| _{L^{\infty }}$,
and $  \| \beta ^{\varepsilon  } (u ^{\varepsilon  }  ) \| _{L^{\infty }} $ independent of $\varepsilon $,
which yields
$u _{\varepsilon } \to u $ strongly in $L^{p} (\Omega )$
and  $\mathcal{A} u ^{\varepsilon  } \to \mathcal{A} u $ $\ast$-weakly in $L^{\infty } (\Omega )$
as  $\varepsilon \to 0$.
Moreover, we obtain $\beta ^{\varepsilon  } (u ^{\varepsilon  }  ) \to  \beta (u ) $
$\ast$-weakly in $L^{\infty } (\Omega )$ by the definition of $\beta ^{\varepsilon }$. 
Therefore the limit $u $ coincides with the unique solution to \eqref{Ellip}
and satisfies $TV _{\Omega } ( \beta  (u ) ) 
\leq  TV _{\Omega } ( h ) $
by the lower semi-continuity of $TV _{\Omega }$
(more precisely, we have 
 $\| \nabla \beta  (u ) \| _{L^1} \leq \| \nabla h \| _{L^1}$
since $\beta  (u ) \in  W^{1, p} _0 (\Omega )$). 

If  $\beta $ is not Lipschitz continuous, we  use this fact to \eqref{Ellip-AP}.
Hence by letting $\lambda \to 0$, we obtain 
\begin{equation}
TV _{\Omega } ( \xi ) \leq
\liminf _{\lambda \to 0} TV _{\Omega } ( \beta _{\lambda } (u _{\lambda } ) ) 
\leq  TV _{\Omega } ( h ) ,
\label{BV-Elliptic} 
\end{equation}
which still holds if $h \in BV _0  (\Omega )$.
\end{Rem}

In order to prove Theorem \ref{MTh-02}, we prepare the following:
\begin{Th}
\label{Th3-2}
Let 
$h _i \in   L^{p' } (\Omega )$
and $(u _i , \xi _i) $ be the solution to (E$_{h _ i} $),
where $ i =1,2$.
Then it holds that
\begin{equation}
\| \xi _1 - \xi _2 \| _{L^1 } \leq \| h _1 - h_2 \| _{L^1}.
\label{Th3-2-1} 
\end{equation}
\end{Th}
\begin{proof}[Proof of Theorem \ref{Th3-2}]
Let $u _{\lambda i }$ ($i =1 ,2 $) be a solution to \eqref{Ellip-AP}
with $h  = h_ i $.
Since 
$\sgn _{\lambda }$ is Lipschitz continuous 
and  \eqref{A-a03} is assumed,
we have
\begin{equation*}
\int_{\Omega } \sgn _{\lambda } (u _{\lambda 1} (x)-u _{\lambda 2} (x) ) 
\LC \mathcal{A} u _{\lambda 1} (x)- \mathcal{A} u _{\lambda 2} (x)  \RC dx \geq 0.
\end{equation*}
Hence
multiplying
the difference of equations
by $\sgn _{\lambda } (u _{\lambda 1} -u _{\lambda 2}  ) $
and taking the limit as $\lambda \to +0 $,
we obtain
\begin{equation}
\int _{\{ x \in \Omega ; ~u _{\lambda 1}(x) \neq u _{\lambda 2}  (x) \} }
 |\beta _{\lambda } (u _{\lambda 1} (x) ) -\beta _{\lambda } (u _{\lambda 2}(x)  ) | dx
 \leq 
 \| h_ 1 - h_2  \| _{L^1}.
\label{PRTh-3-2-01} 
\end{equation}
The L.H.S. of \eqref{PRTh-3-2-01} coincides with
$\|\beta _{\lambda } (u _{\lambda 1}  ) -\beta _{\lambda } (u _{\lambda 2} )  \| _{L^1}$
since $\beta _{\lambda }$ is single-valued.

We here recall that 
there exist a  subsequence of $(u _{\lambda i } , \beta _{\lambda } (u _{\lambda i }))$
which converges to the solution of (E$_{h_i }$).
By the uniqueness of solution to (E$_{h_i }$), that is,
the fact that the limit is determined  independently of the choice of subsequences,
the original sequence $(u _{\lambda i } , \beta _{\lambda } (u _{\lambda i }))$
tends to $(u _i , \xi _i )$ without extracting a subsequence.
Moreover,
by \eqref{Th-E3},
the Dunford-Pettis theorem is applicable to the sequence 
$ \{ \beta _{\lambda } (u _{\lambda 1 }) - \beta _{\lambda } (u _{\lambda 2 }) \} _{\lambda > 0 }$.
Thus
\eqref{Th3-2-1} follows from the limit of \eqref{PRTh-3-2-01} as $\lambda \to 0 $.
\end{proof}

\begin{Rem}
\label{C-L-theory} 
Define a multi-valued operator 
$A _q : L^1 (\Omega ) \to 2 ^{L^1 (\Omega )}$ by
\begin{equation*}
A _q \xi = \mathcal{A} \beta ^{-1} (\xi ) 
=-\nabla \cdot \alpha (\cdot , \nabla  \beta ^{-1} (\xi  (\cdot )) )
\end{equation*}
with domain
\begin{equation*}
D(A _q )
:= \LD 
\xi \in L^{q} (\Omega );~\exists u \in W^{1,p} _0 (\Omega) ~~\text{s.t.}
\begin{matrix}
 ~~\xi (x) \in \beta (u(x))~ ~~\text{a.e.} ~x\in \Omega , \\
~\mathcal{A} u \in L^{q} (\Omega)
\end{matrix}
\RD .
\end{equation*}
Theorem  \ref{Th3-1} implies that $A_q$ fulfills the range condition
\begin{equation*}
D(A _q ) \subset R(\id + \lambda A_{q}) = L^q (\Omega )
\end{equation*}
for any $q \geq p'$
and Theorem  \ref{Th3-2} yields the accretivity of $A_q$.
Therefore we can apply the nonlinear semigroup theory
 and assure that the Cauchy problem
 \begin{equation}
\begin{cases}
~~\DS \frac{d \xi }{dt } + A _q \xi \ni f, \\
~~ \xi (0) = \xi _ 0  ,
\end{cases}
\label{Benilan} 
\end{equation}
which is equivalent to the original problem (P),
possesses a unique integral solution $\xi \in C([0,T] ; L^{1} (\Omega ))$
for any $\xi _0  \in L^1 (\Omega)$ and $f \in L ^{1} (Q)$.
 Furthermore,
 we can see that
 any approximate solution of 
 suitable time-discretization of \eqref{Benilan}
 strongly converges to the integral solution 
 in $L^{\infty} (0,T ; L^1 (\Omega ))$
(see \cite{Be}, \cite{CL}, or Theorem 4.2 of \cite{Bar-1}).
Moreover, by tracing the proof,
we can see that the integral solution satisfies \eqref{MT-03}
if  $f \in \Lip (0,T ; L^1 (\Omega ))$
(remark that it is not trivial that $A _q $ satisfies $L^1$-closedness)
and 
\eqref{external} is not enough to obtain such a time-Lipschitz continuity in abstract setting.
\end{Rem}

\begin{Rem}
By the same procedure as the proof for Theorem \ref{Th3-2},
 we can obtain 
\begin{equation}
\int _{\Omega }  ( \xi _1 (x) - \xi _2 (x) ) _+  dx  \leq \int _{\Omega } ( h _1 (x) - h_2 (x) ) _+ dx,
\label{Th3-2-2} 
\end{equation}
where $(\cdot ) _+ $ denotes the positive part, i.e.,
\begin{equation*}
(s ) _+ :=\begin{cases}
~~s ~~&~~\text{ if } s \geq 0 ,\\
~~0 ~~&~~\text{ if } s < 0 .
\end{cases}
\end{equation*}
Indeed, by testing the difference of \eqref{Ellip-AP} with $h _ i $ by 
$H _{\lambda } ( u_{\lambda 1} - u _{\lambda 2}) $,
we have 
\begin{equation*}
\int _{\{ x \in \Omega ;  u_{\lambda 1}(x) > u _{\lambda 2} (x) \}}
 (\beta _{\lambda } (u_{\lambda 1}(x)) - \beta _{\lambda } (u_{\lambda 2}(x)) ) dx
 \leq 
\int _{\Omega }
(h_1 (x) - h_2(x)) _+ dx  .
\end{equation*}
Since $\beta _\lambda $ is single-valued, this inequality is equivalent to 
\begin{equation*}
\int _{\Omega }
 (\beta _{\lambda } (u_{\lambda 1}(x)) - \beta _{\lambda } (u_{\lambda 2}(x)) )_+ dx
 \leq 
\int _{\Omega }
(h_1 (x) - h_2(x)) _+ dx  .
\end{equation*}
Therefore by lower semi-continuity of $v \mapsto \int _{\Omega} (v ) _+ dx $,
we obtain \eqref{Th3-2-2} by taking the limit as $\lambda \to 0$.
\end{Rem}



\begin{Rem}
We can assure that there exist a pair of initial data $( u_0 , \xi _0 )$
satisfying $\xi _0 (x) \in \beta (u (x))$ and the requirements in Theorem \ref{MTh-01}--\ref{MTh-03}
by solving \eqref{Ellip}with $h $ belonging to $L^{q} (\Omega ) $, $L^{\infty} (\Omega )$ or $BV _0 (\Omega )$.
\end{Rem}


\section{Existence}
Let $N \in \N$ and define $\tau := T / N $.
As the standard approximation of (P),
we consider the following time-discretization problem:
\begin{equation}
\begin{cases}
~~\DS \frac{\xi ^{n+1}_{\tau } (x) - \xi ^{n}_{\tau }(x) }{\tau}
			-\nabla \cdot \alpha (x , \nabla   u^ {n+1} _{\tau } (x) ) = f ^{n} _{\tau}(x) 
			~~&~~x \in \Omega , \\ 
~~\DS \xi ^{n+1}_{\tau } (x) \in \beta (u^{n+1 } _{\tau} (x))
			~~&~~x \in \Omega , \\ 
~~\DS u^{n+1 } _{\tau} (x) =0 
			~~&~~x \in \partial \Omega , 
\end{cases}
\tag{P$_{\tau , n }$}
\label{TDP} 
\end{equation}
where $u ^{0} _{\tau } := u _ 0 $, $\xi ^{0} _{\tau } :=  \xi _0 $, 
and 
$ f ^{n} _{\tau } :=  \frac{1}{ \tau  } \int_{n \tau }^{(n+1) \tau } f (\cdot , s)  ds$.
By using Theorem \ref{Th3-1} with 
$h = \xi ^ n _{\tau } + \tau f ^{n} _{\tau }$,
we can inductively determine
$ u_{\tau } = \{ u ^0 _{\tau } , u ^1 _{\tau } , \ldots , u ^N _{\tau } \} $
and 
$ \xi_{\tau } = \{ \xi ^0 _{\tau } , \xi ^1 _{\tau } , \ldots , \xi ^N _{\tau } \} $ 
satisfying
$u ^n _{\tau } \in W^{1, p} _0 (\Omega ) $, $\mathcal{A} u ^n _{\tau } , \xi ^n _{\tau }
 \in L^{ p'} (\Omega ) \cap L^{ q} (\Omega ) $ for each $n = 1, 2, \ldots , N $, 
and
\begin{equation}
\begin{split}
& \| \xi ^{n+1} _{\tau } \| _{L^{ p'}} \leq \tau \|  f ^n _{\tau } \| _{L^{ p'}} + \| \xi ^{n} _{\tau } \| _{L^{ p'}}
		\leq \tau  \sup _{0\leq t \leq T} \| f( t)  \| _{L^{p' } }+ \| \xi ^{n} _{\tau } \| _{L^{ p'}} , \\
& \| \xi ^{n+1} _{\tau } \| _{L^{ q}} 
		\leq \tau  \sup _{0\leq t \leq T} \| f( t)  \| _{L^{q} }+ \| \xi ^{n} _{\tau } \| _{L^{ q}} , \\
\end{split}
\label{P-01} 
\end{equation}
for $n= 0 ,1, \ldots , N-1$.
Similarly, if $f \in L^{\infty } (0 ,T ; BV_ 0 (\Omega ))$, it holds that
\begin{equation}
 TV _{\Omega } ( \xi ^{n+1} _{\tau }  ) \leq 
\tau  \sup _{0\leq t \leq T} TV _{\Omega } ( f(\cdot , t ) ) 
+ TV _{\Omega } ( \xi ^{n} _{\tau }  ) . 
\label{P-11} 
\end{equation}

By repeating exactly the same procedures as those in \cite{GM} and \cite{R},
we establish a priori  estimates of $u_{\tau}$ and $\xi _{\tau}$.
Multiplying \eqref{TDP} by $u ^{n +1} _{\tau}$,
we have
\begin{equation*}
  \int_{\Omega } j^{\ast} (\xi ^{n+1} _\tau (x)) dx -
\int_{\Omega } j^{\ast} (\xi _0 (x)) dx  +
  c_3 \tau \sum_{m=1}^{n+1}  \| \nabla u^{m} _\tau  \| ^p _{L^p } 
\leq 
 C _3 T \sup _{0\leq t \leq T}\| f(t ) \| ^{p'}_{W^{-1 , p' }}
\end{equation*}
for $n =0, 1,\ldots , N-1$ (recall $n \tau  \leq N \tau = T$).
Here and henceforth,
$c_ 3 , C_3 >0 $ stands for general constants independent of $\tau$ and $N$.
By using
\begin{align*}
&  \int_{\Omega } j^{\ast} (\xi ^{n+1} _\tau (x)) dx -
\int_{\Omega } j^{\ast} (\xi _0 (x)) dx \\
&=
  \int_{\Omega } u ^{n+1} _\tau (x) \xi ^{n+1} _\tau (x) dx -
\int_{\Omega } j(u ^{n+1} _\tau (x)) dx 
-  \int_{\Omega } j^{\ast} (\xi _0 (x)) dx \\
&\geq
 -  \int_{\Omega } j(0) dx -  \int_{\Omega } j^{\ast} (\xi _0 (x)) dx ,
\end{align*}
we obtain
\begin{equation}
\tau \sum_{n=1}^{N}  \| \nabla u^{n} _\tau  \| ^p _{L^p }
\leq C_3 .
\label{Esti-01} 
\end{equation}
Next testing \eqref{TDP} by $(u ^{n +1} _{\tau} -  u ^{n } _{\tau} )$,
we have for $n = 1 , 2, \ldots , N-1$
\begin{align*}
&\int _{\Omega } a (x , \nabla u^{n+1} _\tau (x)) dx
	-\int _{\Omega } a (x , \nabla u _0 (x)) dx  \\
&\leq \sum_{m=0}^{n} \int _{\Omega } f ^m _{\tau}(x)  (u ^{m +1} _{\tau}(x) -  u ^{m } _{\tau}(x) )dx \\
&= \sum_{m=1}^{n} \int _{\Omega } \LC  f ^{m-1} _{\tau}(x) - f ^m _{\tau}(x)\RC    u ^{m } _{\tau}(x) dx 
	+  \int _{\Omega } \LC f ^n _{\tau}(x)  u ^{n +1} _{\tau} (x) - f ^0 _{\tau}(x)  u _0 (x)  \RC  dx \\
&\leq  C_3 \left 
	\| \frac{d f  }{ds} (s) \right \| ^{p'} _{L ^{p'} (0 , T ; W^{-1 ,p'} (\Omega ) )}
	+ \tau  \sum_{m=1}^{N} \|\nabla  u ^{m } _{\tau}\| _{L^p} \\
&\hspace{35mm}  
+\sup _{0\leq t \leq T}\| f (t) \| _{W^{-1, p'}} 
   \LC  \|   \nabla u_0 \| _{L^p} +  \|   \nabla u^{n +1} _{\tau} \| _{L^p}  \RC ,
\end{align*}
which implies
\begin{equation}
\sup _{n = 0 ,1 ,\ldots , N} \| \nabla u^{n} _\tau \|  _{L^p}
 \leq C_3 .
\label{Esti-02} 
\end{equation}
From \eqref{A-a02},
we can deduce $ \| \alpha (\cdot , \nabla u ^n_{\tau }) \| _{L^{p'}} \leq C_3  $ and 
\begin{equation}
\sup _{n = 0 ,1 ,\ldots , N} \| \mathcal{A} u^{n} _\tau \|  _{W ^{-1 ,p '}}
 \leq C_3 .
\label{Esti-03} 
\end{equation}
Immediately, we get 
\begin{equation}
\sup _{n = 0 ,1 ,\ldots , N-1} 
\left \| \frac{\xi ^{n+1} _\tau -  \xi ^{n} _\tau }{\tau}  \right \|  _{W ^{-1 ,p '}}
 \leq C_3 .
\label{Esti-04} 
\end{equation}
Multiplying \eqref{TDP} by $u^{n} _\tau$ and $u^{n+1} _\tau$ again,  
we have
\begin{align*}
&\frac{1}{\tau} \LC \int _{\Omega} j^{\ast} (\xi ^{n} _\tau (x)) dx 
-\int _{\Omega} j^{\ast} (\xi ^{n+1} _\tau (x)) dx 
\RC \\
&~~~~~~~\geq  - \| \mathcal{A} u ^{n+1}  _{\tau } \| _{W^{-1 , p'}} \| \nabla u ^{n} _\tau \| _{L^p}
 -\|  f ^n _{\tau } \| _{W^{-1 , p'}} \| \nabla u ^{n} _\tau \| _{L^p}
\end{align*}
and
\begin{equation*}
\frac{1}{\tau} \LC \int _{\Omega} j^{\ast} (\xi ^{n} _\tau (x)) dx 
-\int _{\Omega} j^{\ast} (\xi ^{n+1} _\tau (x)) dx 
\RC
\leq \|  f ^n _{\tau } \| _{W^{-1 , p'}} \| \nabla u ^{n+1} _\tau \| _{L^p} .
\end{equation*}
From \eqref{Esti-02} and \eqref{Esti-03} it follows that 
\begin{equation}
\sup _{n = 0 ,1 ,\ldots , N-1} 
\LZ  \int _{\Omega} j^{\ast} (\xi ^{n+1} _\tau (x)) dx 
-\int _{\Omega} j^{\ast} (\xi ^{n} _\tau (x)) dx \RZ
 \leq \tau C_3 .
\label{Esti-05} 
\end{equation}

If $\mathcal{A} u _ 0 \in L^{p'} (\Omega )$,
we can define $\xi ^{-1 } _{\tau } \in L^{p'} (\Omega )$ by 
\begin{equation*}
 \xi ^{-1}_{\tau }: =
\xi ^{0}_{\tau }   
			 +\tau \mathcal{A}   u^ {0} _{\tau }   
-\tau f ^{-1} _{\tau}
~~\Leftrightarrow ~~
 \frac{\xi ^{0}_{\tau }  - \xi ^{-1}_{\tau } }{\tau}
			+\mathcal{A}   u^ {0} _{\tau }   = f ^{-1} _{\tau},
\end{equation*}
where $f ^{-1 } _{\tau } \equiv f (\cdot , 0)$.
Then $u ^{0} _{\tau} = u _0 $ and $\xi ^{0} _{\tau} = \xi _0 $
can be regarded as a unique solution to
\eqref{Ellip} with $h = \xi ^{-1}_{\tau }  + \tau f ^{-1 } _{\tau } \in L^{p'} (\Omega )$.
 Theorem \ref{Th3-2} assure that
\begin{align*}
\| \xi ^{n+1} _{\tau} -\xi ^{n} _{\tau}  \| _{L^1} 
&\leq 
\| \xi ^{n} _{\tau} -\xi ^{n-1} _{\tau}  \| _{L^1}
+ \tau \|  f ^{n } _{\tau } -  f ^{n -1 } _{\tau } \|  _{L^1} \\
&\leq 
\| \xi ^{n} _{\tau} -\xi ^{n-1} _{\tau}  \| _{L^1}
+ \tau \int_{(n-1)\tau }^{n\tau } \left  \|  \frac{f (t + \tau)  -f (t ) }{\tau} 
		\right \|  _{L^1} dt
\end{align*}
 for any $n = 0, 1 , \ldots , N -1 $.
Therefore, we obtain
\begin{equation}
\sup _{n= 0, 1, \ldots , N -1 } 
\left \| \frac{\xi ^{n+1} _{\tau} -\xi ^{n} _{\tau}}{\tau }  \right  \| _{L^1} 
\leq 
\left \| f^{-1} _{\tau} - \mathcal{A} u ^ 0 _{\tau}  \right  \| _{L^1} 
+  \int_{0}^{T  } \left  \|  \frac{d f  }{dt } (t)
		\right \|  _{L^1} dt \leq C_3 .
\label{Esti-06}
\end{equation}
We here put
$f ( \cdot ,  t ) \equiv f (\cdot , 0)  $ for $t <0  $.

We now discuss the convergence as $\tau \to 0$. 
If there is no confusion,
a subsequence may be denoted again by the same symbol as the original sequence.
For a given sequence
$w _{\tau }  := \{ w ^{0} _{\tau} ,w ^{1} _{\tau} ,\ldots , w ^{N} _{\tau} \}$,
we set
\begin{align*}
&\Pi _{\tau } w _{\tau } (t):= \begin{cases}
		~~ w^{n+1} _{\tau}~~&~~ \text{ if } t\in (n\tau , (n+1)\tau ] , \\
			~~ w^0_{\tau} ~~&~~ \text{ if } t=0
			\end{cases} \\
&\Lambda  _{\tau}w _{\tau } (t) :=  
\begin{cases}
		\DS ~~ \frac{w ^{n+1} _{\tau} -  w ^{n} _{\tau} }{\tau } (t - n \tau) + w ^{n} _{\tau} 
				~~&~~ \text{ if } t \in ( n\tau , (n+1) \tau ] , \\
			~~ w^0_{\tau} ~~&~~ \text{ if } t=0.
			\end{cases} 
\end{align*}
Then  \eqref{Esti-01}--\eqref{Esti-06} imply
\begin{align}
&\sup _{0\leq t \leq T } \| \Pi _{\tau} u _{\tau } \| _{W^{1,p}_0} 
+\sup _{0\leq t \leq T } \| \mathcal{A} \Pi  _{\tau} u _{\tau } \| _{W^{-1 , p' }} \leq C_3,
\label{P-011}  \\
&	 \sup _{0\leq t \leq T } \| \partial _t \Lambda   _{\tau} \xi  _{\tau } \| _{W^{-1 , p' }} \leq C_3 ,
\label{P-012} \\
&\sup _{0\leq t \leq T } \LZ \partial _t  \Lambda   _{\tau}  \int _{\Omega } j^{\ast } (\xi _{\tau} (x))dx \RZ 
		\leq C_3  ,
\label{P-013} 	
\end{align}
and
\begin{equation}
\begin{split}
&\sup _{0\leq t \leq T } \| \Lambda  _{\tau} \xi _{\tau } \| _{L^{p'}} 
		\leq  T\sup _{0\leq t \leq T }   \| f( t)  \| _{L^{p' } }  + \| \xi _0 \| _{L^{ p'}} ,\\
&\sup _{0\leq t \leq T } \| \Lambda  _{\tau} \xi _{\tau } \| _{L^{q}} 
		\leq   T\sup _{0\leq t \leq T } \| f( t)  \| _{L^{q } }  + \| \xi _0 \| _{L^{ q}} .
\end{split}
\label{P-02} 
\end{equation}
When $\mathcal{A} u _0 \in L^{p'} (\Omega )$, we can derive from \eqref{Esti-06}
\begin{equation}
 \| \Lambda  _{\tau} \xi _{\tau } (t) - \Lambda  _{\tau} \xi _{\tau } (s)  \| _{L^{1}} 
		\leq C_3| t - s | \hspace{5mm} \forall s ,t \in [0,T]
\label{P-03} 
\end{equation}
and 
if $\xi _0 \in BV (\Omega )$ and $f \in L^{\infty} (0,T ;BV_0 (\Omega ) ) $, we obtain 
\begin{equation}
\sup _{0\leq t \leq T } TV _{\Omega } (\Pi _{\tau } \xi_{\tau} ( t) )
	\leq T \sup _{0\leq t \leq T } TV _{\Omega } (f (t) ) 
			+ TV _{\Omega } (\xi _0 ).
\label{P-13} 
\end{equation}
Hence 
we can extract a suitable subsequence 
such that  
we obtain as $\tau \to 0 $
\begin{align*}
\Pi _{\tau} u _{\tau }\to \exists u
		&~~\text{ *-weakly in }L^{\infty } (0,T; W^{1,p} (\Omega )),\\
\Lambda  _{\tau} \xi _{\tau }\to \exists \xi
			&~~\text{ strongly in }C ([0,T]; W ^{-1 , p' } (\Omega )),\\
~
			&~~\text{ *-weakly in }L^{\infty } (0,T; L ^{p'} (\Omega ))
					\cap L^{\infty } (0,T; L ^{q} (\Omega )),\\
\partial _t \Lambda  _{\tau} \xi _{\tau } \to \partial _t \xi  
			&~~\text{ *-weakly in }L^{\infty } (0,T; W ^{-1,p'} (\Omega )), \\
\mathcal{A} \Pi _{\tau} u _{\tau  } \to  \exists \eta  
			&~~\text{ *-weakly in }L^{\infty } (0,T; W ^{-1,p'} (\Omega )),
\end{align*}
as $\tau \to 0$.
Clearly, the limit inferior of \eqref{P-02} as $\tau \to 0$ yields 
\eqref{MT-02}.
If $\mathcal{A} u _0 \in L^{p'} (\Omega )$,
we obtain \eqref{MT-03} by applying Dunford-Pettis's theorem to \eqref{P-03}.
Moreover, if $\xi _ 0 \in BV   (\Omega )$ and $f \in L^{\infty} (0,T ; BV_0 (\Omega) )$,
 \eqref{MT-13} can be derived  from \eqref{P-13}.

According to B\'{e}nilan \cite{Be}
(see also Theorem 4.2 of Barbu \cite{Bar-1}),
the sequence $\{ \Pi _{\tau} \xi _{\tau}\} _{\tau > 0}$
 strongly converges to the integral solution 
 in $L^{\infty} (0,T ; L^1 (\Omega ))$.
 Evidently, the limit coincides with $\xi$.
 Furthermore, since 
\begin{align*}
&\| \Pi _{\tau} \xi _{\tau} (t) - \xi (t) \| ^r _{L^r (\Omega )} \\
& \leq \LC \int _{\Omega } |\Pi _{\tau} \xi _{\tau}(x,t)  - \xi (x,t)| dx \RC ^{1 - 1 / s }
 \LC \int _{\Omega  } |\Pi _{\tau} \xi _{\tau}(x,t)  - \xi (x,t)| ^{p'} dx \RC ^{ 1 / s }
\end{align*}
where $s := 1 / (r -1) (p -1 ) $ belongs to $(1, \infty )$ if $r \in (1 ,p')$,
we have
\begin{equation}
 \Pi _{\tau} \xi _{\tau}  \to  \xi
 ~~~~~\text{ strongly in } L^{\infty} (0,T ; L^{r} (\Omega ) )
 ~~~\forall r \in (1, p ').
\label{Conv-01} 
\end{equation}

Now we  show that
$\eta = \mathcal{A} u $  and $\xi (x , t ) \in \beta (u (x,t))$ for a.e. $ (x, t ) \in Q$
by reprising the argument of \cite{GM}.
By
\eqref{P-012}, \eqref{P-02}
and the compactness of $L^{p'} (\Omega ) \hookrightarrow W^{-1 , p' } (\Omega )$,
Ascoli's theorem is applicable to $\{ \Lambda _\tau \xi_{\tau } \} _{\tau >0 }$ 
and
there is a subsequence 
which strongly converges 
(to $\xi $, obviously) in $C([0,T]; W^{-1, p'} (\Omega ))$.
Remark that
$\Pi _{\tau } \xi _{\tau } -\Lambda  _{\tau } \xi _{\tau } \to 0 $,
i.e.,
$\Pi _{\tau } \xi _{\tau }  \to \xi $ holds
 strongly in 
$L^{\infty } (0,T ; W^{-1, p'} (\Omega ))$
by \eqref{Esti-04}.
Therefore we  obtain
$\LA \Pi _{\tau } \xi _{\tau} , \Pi _{\tau } u _{\tau} \RA _{X} \to
\LA  \xi  ,  u  \RA _{X} $
with $X = L^p (0, T ; W^{1 , p } _0 (\Omega ))$, which implies
\begin{equation*}
\int_{Q}  \Pi _{\tau } \xi _{\tau} (x,t) \Pi _{\tau } u _{\tau} (x,t) dx dt 
\to
\int_{Q}   \xi  (x,t)  u  (x,t) dx dt .
\end{equation*}
Thus Lemma 1.2 in Br\'{e}zis--Crandall--Pazy \cite{BCP}
leads to $\xi \in \widetilde{\beta } (u) $ in $L^{\infty } (0,T ; L^{p'} (\Omega ))$.

By \eqref{P-013},  
Ascoli--Arzela's theorem assures that 
$\{ \Lambda _{\tau} \int _{\Omega} j^{\ast} (\xi _{\tau }) \} _{\tau >0 }$
possesses a subsequence which converges strongly in $C([0,T]; \R)$.
Let $\Theta $ be its limit.
Since $\Pi _{\tau} \xi _{\tau} (T ) \to \xi (T)$
weakly in $L^{p' } (\Omega)$,
we have
$\Theta (T) \geq \int _{\Omega } j ^{\ast} (\xi (T) )dx$
by lower semi-continuity of $\int _{\Omega} j^{\ast} (\cdot ) dx$.
Then by \eqref{L-AL} and
$\Theta (0) =  \int _{\Omega } j^{\ast} (\xi _0 (x)) dx$,
we have
\begin{equation}
\begin{split}
\int_{0}^{T} \LA \partial _t \xi (t) , u (t) \RA  _{W^{1,p} _0} dt 
 &= \int _{\Omega } j^{\ast} (\xi (x, T)) dx - \int _{\Omega } j^{\ast} (\xi _0 (x)) dx, \\
 &\leq \Theta (T) - \Theta (0).
\end{split}
\label{P-05} 
\end{equation}
Multiplying \eqref{TDP} by $u^{n +1} _{\tau }$
and integrating over $(n\tau , (n+1) \tau)$,
we get
\begin{align*}
&\int_{\Omega} j^{\ast} (\xi ^{n+1} _{\tau }(x)) dx -\int_{\Omega} j^{\ast} (\xi ^{n} _{\tau }(x)) dx 
+ \int_{n\tau}^{(n+1)\tau } \LA \mathcal{A} \Pi _{\tau} u_{\tau} (t) , \Pi _{\tau} u_{\tau} (t)  \RA _{W^{1, p} _0} dt\\
&\leq 
\int_{n\tau}^{(n+1)\tau } \int _{\Omega} f  (x, t)  \Pi _{\tau} u_{\tau} (x, t)dxdt
\end{align*}
Calculating $\sum_{n=0}^{N-1} $ and taking its limit as $\tau \to 0$,
we can derive from \eqref{P-05} 
\begin{align*}
&\limsup _{\tau \to 0}
 \int_{0}^{T } \LA \mathcal{A} \Pi _{\tau} u_{\tau} (t) , \Pi _{\tau} u_{\tau} (t)  \RA _{W^{1, p} _0} dt\\
&\leq 
\int_{0}^{T} \int _{\Omega} f  (x, t)   u (x, t)dxdt
-\Theta (T) + \Theta (0) \\
&\leq 
\int_{0}^{T} \int _{\Omega} f  (x, t)   u (x, t)dxdt
- \int _{\Omega } j^{\ast} (\xi (x, T)) dx + \int _{\Omega } j^{\ast} (\xi _0 (x)) dx\\
&=
\int_{0}^{T} \int _{\Omega} 
\LA f (t) - \partial _t \xi (t) ,  u(t)  \RA _{W^{1, p} _0} dt
=
\int_{0}^{T} \int _{\Omega} 
\LA \eta (t) ,  u (t)  \RA _{W^{1, p} _0} dt.
\end{align*}
By virtue of Lemma 1.2 of  \cite{BCP},
we can assure that $\eta = \mathcal{A} u$, whence it follows Theorem \ref{MTh-01}, \ref{MTh-02}, and \ref{MTh-12}.

\begin{Rem}
By the same way as \cite{C}, it can be shown that
\begin{equation}
\Pi _{\tau } u _{\tau} \to u 
 \hspace{5mm}\text{ strongly in } L^{p} (0,T ; W^{1, p} _0  (\Omega ) ).
\label{Conv-02} 
\end{equation}
Indeed, multiplying \eqref{TDP} by $u ^{n+1} _{\tau} - u(t) $ with $t \in (n\tau , (n+1) \tau )$,
we have
\begin{align*}
&\int _{\Omega } f ^n_{\tau } (x,t )  (u ^{n+1} _{\tau} - u(x, t) ) dx \\
&\geq 
\frac{1}{\tau} \LC \int _{\Omega } j^{\ast} (\xi ^{n+1} _{\tau }(x)) dx 
	 - \int _{\Omega } j^{\ast} (\xi ^{n} _{\tau }(x)) dx  	\RC
-\LA \frac{\xi ^{n+1} _{\tau } - \xi ^{n} _{\tau } }{\tau } ,  u( t)  \RA _{W^{1,p} _0} \\
&\hspace{5mm}+  
\LA  \mathcal{A} u ^{n+1} _{\tau} - \mathcal{A} u ,    u ^{n+1} _{\tau} - u( t)  \RA _{W^{1,p} _0}
+
\LA   \mathcal{A} u ,    u ^{n+1} _{\tau} - u( t)  \RA _{W^{1,p} _0},
\end{align*}
which yields
\begin{align*}
&\int _{Q } \Pi _{\tau} f ^n_{\tau } (x,t )  ( \Pi _{\tau} u  _{\tau}(x, t) - u(x, t) ) dx \\
&\geq 
 \LC \int _{\Omega } j^{\ast} (\xi ^{N} _{\tau }(x)) dx 
	 - \int _{\Omega } j^{\ast} (\xi  _{0}(x)) dx  	\RC
+\int_{0}^{T}  \LA \partial _t \Lambda _{\tau } \xi _{\tau } (t) ,  u( t)  \RA _{W^{1,p} _0} dt \\
&+  
\int_{0}^{T}
\LC  \LA  \mathcal{A} \Pi  _{\tau} u _{\tau} (t) - \mathcal{A} u (t),   
		\Pi _{\tau} u _{\tau} (t) - u( t)  \RA _{W^{1,p} _0} 
+
\LA   \mathcal{A} u (t),  \Pi _{\tau}  u _{\tau}(t) - u( t)  \RA _{W^{1,p} _0} \RC dt.
\end{align*}
Taking the limit as $\tau \to 0$ and using \eqref{P-05},
we obtain
\begin{align*}
0&\geq 
 \Theta (T) 
	 -\Theta (0)
+\int_{0}^{T}  \LA \partial _t \xi  (t) ,  u( t)  \RA _{W^{1,p} _0} dt \\
&\hspace{5mm} + 
\limsup _{\tau \to 0} 
\int_{0}^{T} \LA \mathcal{A} \Pi  _{\tau} u _{\tau} (t) - \mathcal{A} u (t),   
		\Pi _{\tau} u _{\tau} (t) - u( t)  \RA _{W^{1,p} _0} dt\\
&\geq 
 \int _{\Omega } j^{\ast} (\xi (x, T)) dx - \int _{\Omega } j^{\ast} (\xi _0 (x)) dx
+\int_{0}^{T}  \LA \partial _t  \xi  (t) ,  u( t)  \RA _{W^{1,p} _0} dt \\
&\hspace{5mm}+ 
\limsup _{\tau \to 0} 
\int_{0}^{T} \LA \mathcal{A} \Pi  _{\tau} u _{\tau} (t) - \mathcal{A} u (t),   
		\Pi _{\tau} u _{\tau} (t) - u( t)  \RA _{W^{1,p} _0} dt \\
&= 
\limsup _{\tau \to 0} 
\int_{0}^{T} \LA \mathcal{A} \Pi  _{\tau} u _{\tau} (t) - \mathcal{A} u (t),   
		\Pi _{\tau} u _{\tau} (t) - u( t)  \RA _{W^{1,p} _0} dt .		
\end{align*}
Hence \eqref{Conv-02} follows from \eqref{A-a03}.
\end{Rem}

In order to prove Theorem \ref{MTh-03}, we use Lemma 4 of Carrillo \cite{C}:
let 
$(u, \xi)$ fulfill \eqref{Regu} 
and $\gamma : \R \to \R$ be non-decreasing Lipschitz continuous.
If
$\zeta \in C ^{1 } (\overline{Q })$ satisfies $\gamma (u) \zeta \in L^{\infty} (0, T ; W^{1,p } _0 (\Omega ))$,
then 
\begin{equation}
\begin{split}
&\int _{\Omega } \Gamma  (\xi (x,t ) ) \zeta (x ,t ) dx  
 - \int _{\Omega } \Gamma  (\xi _0 (x ) ) \zeta (x , 0 ) dx \\ 
&~~~
=\int_{0}^{t}  \LA \partial _t \xi (t), \gamma (u(t)) \zeta (t)  \RA _{W^{1,p} _0 } dt
 + \int_{0}^{t} \int _{\Omega } \Gamma (\xi (x,t )) \partial _t \zeta (x,t) dx dt 
\end{split}
\label{Lem-Cal} 
\end{equation}
holds for a.e. $t \in (0,T)$, where
\begin{equation}
\Gamma (s) :=
\begin{cases}
~~\DS \int_{0}^{s} \gamma ( \LC \beta ^{-1} \RC ^{\circ} ( \sigma )) d\sigma     ~~&~~
		\text{ if } s \in D(\gamma \circ  \beta ^{-1} ), \\
~~ +\infty ~~&~~
		\text{ otherwise}.
\end{cases}
\label{Lem-Cal-Gam} 
\end{equation}
We use this formula with $\zeta \equiv 1$ and 
$ \gamma = K ^{p(r-1) +2 }_M $.
Since 
\begin{equation*}
0 \leq \Gamma (\xi (x,t )) \leq \xi (x,t) K ^{p(r-1) +2} _M (u (x,t )),
\end{equation*}
we can obtain
\begin{align*}
c r ^{-p } \| u \| ^{pr}_{L^{pr \mu } (Q)}
& \leq
\| f \| _{L^{\infty } (Q)} \| u  \| ^{p (r-1) +1} _{ L ^{p (r-1) +1} (Q)}
+ \int _{\Omega } \Gamma  (\xi _ 0  (x)) dx \\
& \leq
\| f \| _{L^{\infty } (Q)} \| u  \| ^{p (r-1) +1} _{ L ^{p (r-1) +1} (Q)}
+\| \xi _0  \| _{L^{\infty } (\Omega )}  
\| u _0   \| ^{p (r-1) +1} _{ L ^{p (r-1) +1} (\Omega )} ,
\end{align*}
with some $\mu > 1 $. Therefore by Morse's iteration, 
we can deduce
\begin{equation*}
\| u \| _{L ^{\infty } (Q)} \leq C =
C (\| f \| _{L^{\infty } (Q)} , \| u _0   \| _{L^{\infty} (\Omega )} ,\| \xi _0   \| _{L^{\infty} (\Omega )} , |\Omega |,
|Q| ).
\end{equation*}


\section{Uniqueness} 
Throughout this section,  we assume that
\begin{equation}
\alpha (x , z ) = \alpha ( z ) 
\label{Unique} 
\end{equation}
and initial data and external force satisfy
\begin{equation}
\begin{split}
& f \in W^{1, p' } (0,T ; W^{-1, p'} (\Omega )) \cap L^{\infty} (0,T; L^{p'} (\Omega ) ) ,\\
& u _{0} \in W^{1,p} _0 (\Omega ),~~\xi_0  \in L^{p'} (\Omega ) ,
~~\xi _0(x) \in \beta (u_0 (x))~\text{ for a.e. }x\in \Omega .
\end{split}
\label{Assume2} 
\end{equation}
To discuss the uniqueness of solution, we rewrite our problem.
Let $v := u + \xi $, where $(u, \xi )$ is a solution to (P).
Since $\beta $ and $\beta ^{-1}$ are maximal monotone,
$g := (\id + \beta ^{-1}) ^{-1} $ and $b := (\id + \beta ) ^{-1}$
are Lipschitz continuous 
and satisfy $ g(v) = \xi $, $b (v) = u$, and $g \circ b ^{-1} = \beta $, $b \circ g ^{-1} = \beta ^{-1} $.
In this manner,
we obtain the following, which is equivalent to the 
 original initial boundary value problem of (P):
\begin{equation}
\begin{cases}
~ \partial _t g(v (x,t)) - \nabla \cdot  \alpha  ( \nabla b( v (x, t) )) = f(x,t) 
		~&(x,t) \in Q , \\
~ b (u (x,t)) = 0 
		~~~&(x,t) \in \partial \Omega \times  (0,T), \\
~ g(v (x,0)) = \xi _0(x) 
		~~~&x \in \Omega .
\end{cases}
\label{E-001} 
\end{equation}
Theorem \ref{MTh-01} assures the existence of a solution to \eqref{E-001} in the following sense:
\begin{De}
\label{E-weak} 
A function $v \in L^{1} (Q)$  is said to be a weak solution to \eqref{E-001}
if
\begin{equation}
\begin{split}
&b(v) \in L^{\infty } (0,T ; W^{1,p } _0 (\Omega )), \\
&g(v) \in W^{1 ,\infty } (0,T ; W^{-1,p '}  (\Omega )) \cap L^{\infty } (0,T ; L^{p ' }  (\Omega )), \\
&\partial _t g(v) - \nabla \cdot \alpha (\nabla   b( v )  ) = f ~
\text{ in }W^{-1, p'} (\Omega )~~~\text{ for a.e. } t  \in (0,T) , \\
& g(v (\cdot ,0)) = \xi _0(\cdot ) .
\end{split}
\label{E-Regu} 
\end{equation}
\end{De}
\noindent We here introduce the definition of entropy solution 
\begin{De}
\label{E-entro} 
A weak solution $v$ to \eqref{E-001} is said to be a entropy solution 
if
\begin{align}
&\int _{Q} H^{\circ } (v  - s) 
\LD \alpha ( \nabla b (v )) \cdot \nabla \zeta  - (g (v ) - g(s)) \partial _t \zeta - f \zeta  \RD dx dt
		\label{E-Def1} \\
&\hspace{60mm}  - \int _{\Omega } (\xi _0 - g (s)) _+ \zeta (\cdot , 0) dx
		\leq 0,
			\notag \\
&\int _{Q} H^{\circ} ( - s - v ) 
\LD \alpha ( \nabla b (v )) \cdot \nabla \zeta  - (g (v ) - g(-s)) \partial _t \zeta  
		- f  \zeta  \RD dx dt 
		\label{E-Def2} \\
&\hspace{60mm}+ \int _{\Omega } ( g (-s) - \xi _0 ) _+ \zeta (\cdot , 0) dx
		\geq 0,
			\notag 
\end{align}
for any $s\in \R$ and $ \zeta = \zeta (x,t) $ such that $\zeta \geq 0$ and  either of
\begin{equation}
\text{i)  } s \geq 0 ~~\text{ and }~~\zeta \in \mathscr{D} ([ 0, T ) \times \overline{\Omega } ),
~~~~\hspace{5mm} 
\text{ii)  } s \in \R ~~\text{ and }~~\zeta \in \mathscr{D} ([ 0, T ) \times \Omega  ).
\label{Alter} 
\end{equation}
\end{De}

Main assertion of this section is  that
the solution constructed in the previous section
meets the requirements of entropy solution.
To this end, we  return to the elliptic problem:
\begin{Le}
\label{E-E-Lem} 
Assume (H.$\alpha $), (H.$\beta $), \eqref{Unique}, and $h \in L^{p'} (\Omega )$.
Let $w \in L^1 (\Omega )$ be a solution to
\begin{equation}
\begin{cases}
~~ g (w(x))  -\nabla \cdot \alpha ( \nabla b(w(x)) ) = h(x)~~~~&~~x \in \Omega ,\\
~~ b(w(x)) = 0   ~~~~&~~x \in \partial \Omega ,
\end{cases}
\label{E-Ellip} 
\end{equation}
such that $g (w) ,  \mathcal{A} b (w) \in L^{p'} (\Omega )$ and $b(w) \in W^{1 , p} _0 (\Omega )$.
Then $w$ is a entropy solution to \eqref{E-Ellip},
i.e., it holds that
\begin{equation}
\int _{\Omega } H^{\circ} (w -s ) 
\LD (g (w ) - h  )\zeta  
+  \alpha ( \nabla b(w ) )\cdot \nabla \zeta  \RD dx  \leq 0
\label{E-Lem001} 
\end{equation}
and 
\begin{equation}
\int _{\Omega } H^{\circ} (- w  -s ) 
\LD (g (w  ) - h )\zeta 
+  \alpha ( \nabla b(w )) \cdot  \nabla \zeta  \RD dx  \geq 0
\label{E-Lem002} 
\end{equation}
for every 
$(s , \zeta )$  satisfying $\zeta \geq 0$ and either of 
\begin{equation}
\text{i)  } s \geq 0 ~~\text{ and }~~\zeta \in \mathscr{D} ( \overline{\Omega }),
~~~~\hspace{5mm} 
\text{ii)  } s \in \R ~~\text{ and }~~\zeta \in  \mathscr{D} ( \Omega ).
\label{Alter2} 
\end{equation}
\end{Le}
\begin{proof}
Remark that 
 if $w$ is a solution to \eqref{E-Ellip},
then $-w$ becomes a solution
    to \eqref{E-Ellip} with 
$g $ replaced by 
$\hat{g} (\sigma ) := - g (- \sigma )$,
$b $ by 
$\hat{b} (\sigma ) := - b (- \sigma )$
(namely $\beta  $ by 
$\hat{\beta } (\sigma ) := - \beta (- \sigma )$, which is still maximal monotone),
$\alpha  $ by 
$\hat{\alpha } ( z ) := -\alpha  ( - z )$,
and $h $ by $- h$.
Hence 
it is sufficient to prove \eqref{E-Lem001}.

Multiplying
\eqref{E-Ellip} by $H_{\lambda } (b (w) - b (s)) \zeta \in W^{1,p} _0 (\Omega )$ 
with $(s, \zeta )$ satisfying i) or ii) in Lemma \ref{E-E-Lem}
and letting $\lambda \to +0$,
we get 
\begin{equation*}
\int_{\Omega } H^{\circ } (b (w) -b (s))
\LD  (g (w) - h ) \zeta + \alpha (\nabla b (w) ) \cdot \nabla \zeta \RD  dx \leq 0 .
\end{equation*}
Here we assume that $b (s) \not \in E$, where 
\begin{equation}
E :=  \{  s \in R (b);~~b ^{-1} (s) \text{ is multi-valued} \} 
=  \{  s \in D (\beta );~~\beta (s) \text{ is multi-valued} \} .
\label{null-set} 
\end{equation}
 Since $b (\sigma ) =b(s)$ is attained only by $\sigma = s $,
 we have  $H_{0} (w (x,t ) - s) = H_0 (b(w (x,t )) - b(s))$ for a.e. $(x, t) \in Q$,
which immediately leads to \eqref{E-Lem001}.
Especially, if $\beta $ is single-valued, 
$E$ is empty and \eqref{E-Lem001} holds for every $s $.

To prove the general case, we consider the following approximate problem:
\begin{equation}
\begin{cases}
~~ g \circ b ^{-1}_{\varepsilon }  (u _{\varepsilon }(x)) 
		-\nabla \cdot \alpha ( \nabla u_{\varepsilon } (x) ) = h _n (x)~~~~&~~x \in \Omega ,\\
~~ u_{\varepsilon } (x) = 0   ~~~~&~~x \in \partial \Omega ,
\end{cases}
\label{E-Ap01} 
\end{equation}
where $b _{\varepsilon } = \varepsilon  \id + b $ with $\varepsilon >0 $
and $\{ h_ n \} _{n \in \N }$ is a sequence in $\mathscr{D} (\Omega )$ converging to $h$ in $L^{p'} (\Omega )$.
By  Theorem \ref{Th3-1},
there is a unique solution 
$u_{\varepsilon }  \in W^{1,p} _0 (\Omega )$ such that 
$\xi _{\varepsilon } : = g \circ b ^{-1}_{\varepsilon }  (u _{\varepsilon }) , 
		\nabla \cdot \alpha ( \nabla u _{\varepsilon })  \in L^{p' } (\Omega )$.
Moreover,
by a priori estimates and Remark \ref{Estimate-Ellip} (use here \eqref{Unique}),
we have 
\begin{equation*}
\| \xi _{\varepsilon } \| _{L^{\infty }} ,~
\| \xi _{\varepsilon } \| _{W^{1, 1}} ,~
\| u _{\varepsilon } \| _{W ^{1,p} } ,~ 
\| u _{\varepsilon } \| _{L^{\infty }} \leq C _n,
\end{equation*}
where $C_n $ is the  general constant which is independent of $\varepsilon > 0$.
By Rellich-Kondrachov's theorem, we can extract a subsequence 
such that
\begin{align*}
u _{\varepsilon } \to \exists u_n ~~&~~\text{strongly in }L^{p} (\Omega ) ,\\
				~~&~~\text{weakly in }W^{1,p }_0 (\Omega ) ,\\
				~~&~~\ast \text{-weakly in }L^{\infty } (\Omega ) ,\\
\xi _{\varepsilon } \to \exists \xi _n ~~&~~\text{strongly in }L^{r} (\Omega ) ~~\text{ for any } r \in (1,\infty ),\\
				~~&~~\ast \text{-weakly in }L^{\infty }(\Omega ) ,\\
\nabla \cdot \alpha ( \nabla  u _{\varepsilon } ) \to \exists \eta _n 
		~~&~~\text{strongly in }L^{r} (\Omega ) ~~\text{ for any } r \in (1,\infty ),\\
				~~&~~\ast \text{-weakly in }L^{\infty }(\Omega ) .				
\end{align*}
Let $w _{\varepsilon } := b _{\varepsilon } ^{-1} (u _{\varepsilon }  )$,
then $  \| g(w  _{\varepsilon } ) \| _{L^{\infty }} \leq C_n  $
and   $ \| b (w _{\varepsilon } ) \| _{L^{\infty }} \leq  \| u  _{\varepsilon }  \| _{L^{\infty }} \leq C_n  $.
Remark that 
$g + b $ is surjective
and $w _{\varepsilon } = g (w _{\varepsilon } ) + b (w _{\varepsilon } )$ holds 
by the maximal monotonicity of $\beta $.
Hence 
\begin{align*}
w _{\varepsilon } \to \exists w_n  ~~&~~\text{strongly in }L^{p} (\Omega ) ,\\
				~~&~~\ast \text{-weakly in }L^{\infty } (\Omega ) , \\
\varepsilon w _{\varepsilon } \to 0  ~~&~~\text{strongly in }L^{\infty } (\Omega ) .
\end{align*}
Since $b$ and $g $ are Lipschitz continuous,
we obtain $u_n = b (w _n )$, $\xi _n = g (w _n )$,  $w _n =  g (w _n ) + b (w _n )$,
and $\eta _n = \nabla \cdot \alpha ( \nabla  u _n )$.
Testing the difference of \eqref{E-Ap01} by 
$b _{\varepsilon _1} (w _{\varepsilon _1}) -b _{\varepsilon _2} (w _{\varepsilon _2})$
and using \eqref{A-a03},
we have 
\begin{equation*}
c \| \nabla b _{\varepsilon _1} (w _{\varepsilon _1}) -\nabla b _{\varepsilon _2} (w _{\varepsilon _2})  \| ^p_{L^p } 
\leq 
 (\varepsilon _2 - \varepsilon _1) \int_{\Omega }(g (w _{\varepsilon _1}) -g (w _{\varepsilon _2}) )
  w _{\varepsilon _2 }   dx , 
\end{equation*}
which yields $b _{\varepsilon } (w _{\varepsilon }) \to b  (w _n)$
strongly in $W^{1,p} _0 (\Omega )$.
By Lebesgue's dominated convergence theorem and continuity of $\alpha $, we have 
$\alpha ( \nabla b _{\varepsilon } (w _{\varepsilon }) ) \to \alpha ( \nabla b(w _n) ) $
strongly in $L^{p'} (\Omega )$.
Therefore, the limit of solution to \eqref{E-Ap01} as $\varepsilon \to 0 $
coincides with a unique  solution to 
\begin{equation}
\begin{cases}
~~ g (w _n(x)) 
		-\nabla \cdot \alpha ( \nabla b  (w_n (x)) ) = h _n (x)~~~~&~~x \in \Omega ,\\
~~ b(w_n (x)) = 0   ~~~~&~~x \in \partial \Omega .
\end{cases}
\label{E-Ap02} 
\end{equation}
Moreover, since $g \circ b ^{-1} _{\varepsilon }$ is single-valued,
$w _{\varepsilon }$ is an entropy solution to \eqref{E-Ap01}, i.e.,
\begin{equation*}
\int _{\Omega } H^{\circ} (w _{\varepsilon } -s ) 
\LD (g (w _{\varepsilon }) - h _n  )\zeta  
+  \alpha ( \nabla b _{\varepsilon }(w _{\varepsilon }) )\cdot \nabla \zeta  \RD dx  \leq 0
\end{equation*}
holds for every $(s , \zeta )$ in Lemma \ref{E-E-Lem}.
By the maximal monotonicity of $\widetilde{H}$ and strong convergence of $w _{\varepsilon }$ in $L^p (\Omega )$,
there is a subsequence of $H^{\circ } (w _{\varepsilon } - s)$  which $\ast $-weakly 
converges in $L^{\infty } (Q )$ and its limit $ \chi _s $ belongs to $H (w _ n - s )$ for a.e. $(x,t ) \in Q $.
Hence we have for any $s$
\begin{equation*}
\int _{\Omega } \chi _s
\LD (g (w _n) - h _n  )\zeta  
+  \alpha ( \nabla b (w _n) )\cdot \nabla \zeta  \RD dx  \leq 0 .
\end{equation*}
Here put $\{ \sigma  _{i} \} _{i\in \N}$ 
such that $\sigma _{i}  \searrow s $  as $i \to \infty $.
 Then $\chi _{\sigma _i} \in H (w _ n - \sigma _i ) $ satisfies 
 $\chi _{\sigma _i}  \nearrow H ^{\circ} (w _ n - s )$ for a.e. $(x,t) \in Q$. 
 Hence we can derive 
 \begin{equation}
\int _{\Omega } H^{\circ} (w _n -s  )
\LD (g (w _n) - h _n  )\zeta  
+  \alpha ( \nabla b (w ) )\cdot \nabla \zeta  \RD dx  \leq 0 
\label{E-Lem003} 
\end{equation}
for any $(s, \zeta )$, i.e., $w _n $ is an entropy solution to \eqref{E-Ap02}.

By repeating a priori estimates given in Section 3,
we get 
\begin{equation*}
 \| g (w _ n ) \| _{L^{p'}} +  \| b (w _ n ) \| _{W^{1 ,p}} + 
  \| \nabla \cdot \alpha ( \nabla b (w _ n ) ) \| _{L^{p' }}   \leq C
\end{equation*}
and 
\begin{equation*}
 \| \nabla b (w _ n )   -  \nabla  b (w _ m )  \| ^{p-1}_{L^p}
   \leq C \| h _n -h_ m  \| _{L^{p'}}
\end{equation*}
for any $n, m \in \N$,
where $C$ is suitable constant independent of $n$.
Moreover, Theorem \ref{Th3-2} yields
\begin{equation*}
\|  g (w _ n ) -  g (w _ m ) \| _{L^1} +
\|  \nabla \cdot \alpha (\nabla b (w _ n ) )  -  \nabla \cdot \alpha (\nabla b (w _ m ) ) \| _{L^1} 
\leq 3 \| h _n -h_ m  \| _{L^1}.
\end{equation*}
Hence we can see that
\begin{align*}
w _n \to \exists w ~~&~~\text{strongly in }L^{r} (\Omega ) \hspace{3mm} \forall r \in (1, \min \{ p' , p \} ),\\
b (w _n) \to b( w) ~~&~~\text{strongly in }W^{1,p }_0 (\Omega ) ,\\
g (w _n) \to g ( w) ~~&~~\text{strongly in }L^{1 } (\Omega ) \text{ and weakly in }L^{p'} (\Omega ) ,\\
\alpha ( \nabla b (w _n) ) \to \alpha ( \nabla  b( w)) ~~&~~\text{strongly in }L^{p' } (\Omega ) ,\\
\nabla \cdot \alpha ( \nabla b (w _n) ) \to 
\nabla \cdot  \alpha ( \nabla  b( w)) ~~&~~\text{strongly in }L^{1 } (\Omega ) \text{ and weakly in }L^{p'} (\Omega ) ,
\end{align*}
and $w $ is a unique solution to \eqref{E-Ellip}.
By taking the limit of \eqref{E-Lem003} as $n\to \infty $ and repeating the argument above,
we can assure that $w $ satisfies \eqref{E-Lem001}.
\end{proof}

Let  $u _{\tau} := \{ u ^0 _{\tau} , u ^1 _{\tau} , \ldots , u ^N _{\tau} \}$
and $u _{\tau} := \{ \xi ^0 _{\tau} , \xi ^1 _{\tau} , \ldots , \xi ^N _{\tau} \}$
be sequences determined by  \eqref{TDP}
and define $v _{\tau }$ by $v ^n _{\tau} := u ^n _{\tau}  + \xi ^n _{\tau}  $
(remark that $\xi ^n _{\tau}  = g (v ^n _{\tau} ) $ and $u ^n _{\tau}  = b (v ^n _{\tau} ) $).
Note that \eqref{Conv-01} and \eqref{Conv-02} imply
\begin{equation}
 \Pi _{\tau} v _{\tau}  \to  v = u + \xi
 ~~~~~\text{ strongly in } L^{p} (0,T ; L^{r} (\Omega ) )
 \hspace{3mm} \forall r \in (1, \min \{ p , p ' \} ).
\label{Conv-03} 
\end{equation}
Hence 
applying Lemma \ref{E-E-Lem} to \eqref{TDP},
and using \eqref{Conv-01} \eqref{Conv-02} \eqref{Conv-03},
we have (let $\Pi _\tau v _{\tau } (t) = v _ 0  $ for $t < 0$)
\begin{align*}
0 &\geq 
\int _{Q} H^{\circ} (\Pi _\tau v _{\tau } (x, t) -s ) 
\LD \alpha (\nabla \Pi _{\tau } b (v _{\tau } (x, t))) \cdot \nabla \zeta (x,t)  \right .\\
&\hspace{5mm} 
\left .
+  \LC \frac{g (\Pi _{\tau} v _{\tau} (x, t)) - g (\Pi _{\tau} v _{\tau} (x, t-\tau))  }{\tau} - \Pi _{\tau } f _{\tau} \RC
\zeta (x, t)  
\RD dx dt \\
&\geq 
\int _{Q} H^{\circ} (\Pi _\tau v _{\tau } (x, t) -s ) 
\LC \alpha (\nabla \Pi _{\tau } b (v _{\tau } (x, t)))
			\cdot \nabla \zeta (x,t)  
-  \Pi _{\tau } f _{\tau}\zeta (x, t) 
\RC dx dt \\
&\hspace{3mm} + \frac{1}{\tau }  \int_{0 }^{\tau } \int _{\Omega }
  H^{\circ} (\Pi _\tau v _{\tau } (x, t) -s )\LC  g (\Pi _{\tau } v_{\tau } (x,t )) -g(s) \RC \zeta (x ,t ) dx dt
\\  
&\hspace{3mm} - \frac{1}{\tau }  \int_{ -\tau  }^{0 } \int _{\Omega }
  H^{\circ} (\Pi _\tau v _{\tau } (x, t) -s )\LC  g (\Pi _{\tau } v_{\tau } (x,t )) -g(s) \RC \zeta (x ,t + \tau ) dx dt
\\  
&\hspace{3mm} +  \int_{T- \tau  }^{T  } \int _{\Omega }
  H^{\circ} (\Pi _\tau v _{\tau } (x, t) -s )\LC  g (\Pi _{\tau}  v_{\tau } (x,t )) -g(s) \RC 
\frac{  \zeta (x ,t ) -\zeta (x ,t +\tau ) }{ \tau }  dx dt
\\  
& \to 
\int _{Q} \chi _{v, s } 
\LC \alpha (\nabla  b (v  (t))) \cdot \nabla \zeta - \LC  g (v) -g(s) \RC 
\partial _t \zeta  -f \zeta  \RC dx dt \\
&\hspace{75mm} 
- 
 \int _{\Omega }
 \LC  \xi _0  -g(s) \RC  _+ \zeta (0 ) dx   
\end{align*}
By putting $\{ \sigma  _{i} \} _{i\in \N}$ 
such that $\sigma _{i}  \searrow s $  as $i \to \infty $,
we obtain \eqref{E-Def1} for every $ (s ,\zeta )$ in Definition \ref{E-entro}.
Immediately,
we have \eqref{E-Def2} by replacing 
$v$ with $-v $,
which is a solution to  \eqref{E-001}
with 
$\hat{\beta } (\sigma ) := - \beta  (- \sigma )$
substituted for $\beta $,
$\hat{\alpha  } (z ) := - \alpha  (- z )$ for $\alpha  $,
$\hat{f  }  := -f$ for $f $,
and 
$\hat{\xi   } _0   := -\xi _0$ for $\xi _0 $.
Thus  it follows that
\begin{Th}
Assume (H.$\alpha $), (H.$\beta $), and \eqref{Unique}.
Then  \eqref{E-001} possesses at least one entropy solution.
\end{Th}

If $\beta  $ is single-valued, 
we can show that any weak solution to \eqref{E-001} is entropy solution
by  the following lemma:
\begin{Le}
\label{E-Lem-Single} 
Assume (H.$\alpha $), (H.$\beta $),
and let $v$ be a weak solution to \eqref{E-001}.
Then $v$ satisfies 
\begin{equation}
\begin{split}
&\int _{Q} H^{\circ } (v  - s) 
\LD \alpha (x , \nabla b (v )) \cdot \nabla \zeta  - (g (v ) - g(s)) \partial _t \zeta - f \zeta  \RD dx dt
		 \\
&~~~ ~~~- \int _{\Omega } (\xi _0 - g (s)) _+ \zeta (\cdot , 0) dx \\
&~~~ ~~~~~~ ~~~
= - \lim_{\lambda \to 0 } \int_{Q} \alpha (x , \nabla b(v)) \cdot \nabla b(v) H' _{\lambda } (b (v) -b (s))\zeta dx dt ,
\end{split}
\label{E-Lem004}
\end{equation}
for any $(s , \zeta )  $ which fulfills $b(s) \not \in E$,
$\zeta \geq 0$, and either of i) or ii) in \eqref{Alter}
and 
\begin{equation}
\begin{split}
&\int _{Q} H^{\circ } (- s - v ) 
\LD \alpha (x , \nabla b (v )) \cdot \nabla \zeta  - (g (v ) - g(-s)) \partial _t \zeta - f \zeta  \RD dx dt
		 \\
&~~~ ~~~+ \int _{\Omega } ( g (-s ) - \xi _0 ) _+ \zeta (\cdot , 0) dx \\
&~~~ ~~~~~~ ~~~
=  \lim_{\lambda \to 0 } \int_{Q} \alpha (x , \nabla b(v)) \cdot \nabla b(v) H' _{\lambda } (b (-s) -b (v) )\zeta  dx dt ,
\end{split}
\label{Entro-Lem002}
\end{equation}
for any $(s , \zeta ) $ which fulfills $-b (-s) \not \in E$,
$\zeta \geq 0$, and either of i) or ii)  in \eqref{Alter}.
\end{Le}
\begin{proof}
It is sufficient to prove \eqref{E-Lem004}
for any  $(s , \zeta )  $ such that $b(s) \not \in E$.
Multiplying \eqref{E-001} by $H_{\lambda } (b(v) - b (s) ) \zeta \in L^{\infty} (0,T ; W^{1,p} _0 (\Omega ))$
and applying Lemma 4 of \cite{C} (recall \eqref{Lem-Cal}),
we have 
\begin{equation}
\begin{split}
& \int_{Q} \Gamma (g (v )) \partial _t \zeta dxdt + \int_{\Omega } \Gamma (\xi _0 )  \zeta (0) dx \\
 =& \int_{Q} \alpha (x , \nabla  b (v)) \cdot \nabla \LC H_{\lambda } (b(v) - b (s) ) \zeta  \RC dx  dt 
-   \int_{Q} f  H_{\lambda } (b(v) - b (s) ) \zeta  dx   dt  ,  
\end{split}
\label{E-Lem005} 
\end{equation}
where 
$\Gamma $ is defined by \eqref{Lem-Cal-Gam}
with $\gamma (\sigma ) := H_{\lambda } (b(v) - b (s) ) $.
 By $\beta ^{-1} = b \circ g ^{-1}$, we get 
 \begin{align*}
\Gamma (g(v)) &:= \int_{0}^{g(v) } H_{\lambda } ( \LC  \beta ^{-1} \RC ^{\circ } (\sigma ) - b (s)   ) d \sigma 
= \int_{g(s)}^{g(v) } H_{\lambda } ( \LC  \beta ^{-1} \RC ^{\circ } (\sigma ) - b (s)   ) d \sigma \\
&\to ( g (v) -g (s)) _+  = H_{0} (v -s ) (g (v) - g (s)) ,\\
  \Gamma (\xi _0 ) 
& = \int_{g(s)}^{\xi _0 } H_{\lambda } ( \LC  \beta ^{-1} \RC ^{\circ } (\sigma ) - b (s)   ) d \sigma \\
& \to (\xi _0 -g (s)) _+  
\end{align*}
Therefore by letting $\lambda \to 0$,
we obtain \eqref{E-Lem004}
 since
$H_{0} (v (x,t ) - s) = H_0 (b(v (x,t )) - b(s))$ holds by  $b(s) \not \in E$.
\end{proof}


Consequently, we can derive the following,
which implies the comparison principle and the uniqueness of entropy solution.
\begin{Th}
Assume (H.$\alpha $), (H.$\beta $), and \eqref{Unique}.
Let $v _ i$ ($i=1 ,2$) be a entropy solution to
\eqref{E-001}
with 
$v _ i (\cdot , 0)= v _{i0} $, $\xi _{0i} = g(v _ {i0} ) $,
 $u _{0i} = b(v _ {i0} ) $
and $f _i  $
satisfying \eqref{Assume2}.
Then it holds that 
\begin{equation}
\begin{split}
&\int _{\Omega } (g (v _1 (x,t)) - g (v _2 (x,t))) _+ dx \\
&\hspace{15mm} 
\leq \int _{\Omega } (\xi_{10} (x) - \xi_{20} (x)) _+ dx +
	\int _{0}^{t} \int _{\Omega } (f _1 - f_ 2) _+ dx dt
\end{split}
\label{comperi}
\end{equation}
for every $t \in [0,T]$, and therefore
\begin{equation*}
 \|g (v _1 (t)) - g (v _2 (t)) \| _{L^1} 
\leq \| \xi_{10}  - \xi_{20} \| _{L^1 } +
	\int _{0}^{t} \| f _1 (t)- f_ 2 (t) \| _{L^1} dt
\end{equation*}
for every $t \in [0,T]$.
\end{Th}
\begin{proof}
Let $(y , s) \in Q $ and $(x , t) \in Q $ be variables used for $i =1 $ and $2$, respectively.
Namely, set $v _1 = v _1 (y , s)$, $f _1 = f _1 (y , s)$,
and $v _2 = v _2 (x , t)$, $f _2 = f _2 (x , t)$.
Define a non-negative smooth function $Z _1= Z_1 ( y ,s , x,t  )$ on
$Q \times Q$ such that
\begin{align*}
(y ,s ) \mapsto Z _1 ( y ,s ,x,t  ) \in \mathscr{D} ([0, T ) \times \Omega ) 
		~~~~&\text{ for each fixed } (x ,t ) \in Q ,\\
(x ,t ) \mapsto Z _1 (y ,s ,x,t  ) \in \mathscr{D} ([0, T ) \times \Omega ) 
		~~~~&\text{ for each fixed } (y ,s ) \in Q .		
 \end{align*}
Then by the definition of entropy solution
and Lemma \ref{E-Lem-Single}, 
we have
\begin{align*}
&\int _{Q \times Q} H ^{\circ } (v_ 1  - v _2 )
	\LD  -(g (v _1 )  -g (v _2 ) ) (\partial _t Z _1+ \partial _s Z_1)
	\right . \\
&\hspace{5mm} 
\left .
+ (\alpha ( \nabla  _x b (v _1 )) - \alpha ( \nabla  _y b (v _2 )) ) 
		\cdot (\nabla _x Z _1+ \nabla _y Z_1 )
				 - (f _1 - f_2  ) Z_1 \RD dyds dxdt \\
&\hspace{15mm} 
 -\int_{Q} \int _{\Omega }
	(\xi _{01} (y) -g (v _2 (x ,t )) ) _+  Z _1(y, 0 , x ,t  ) dy dx dt \\
&\hspace{15mm} 
 -\int_{Q} \int _{\Omega }
	(g (v _1 (y ,s)) - \xi _{02} (x)  ) _+   Z _1(y ,s, x , 0 ) dy dx ds   \\
& \leq 
- \lim_{\lambda \to 0} 
\int_{ (Q \setminus Q _1 ) \times (Q \setminus Q _2 ) }  
 		H ' _{\lambda } ( b (v _1 ) - b (v _2 )) Z_1 \\
&\hspace{20mm} \times 
(\alpha ( \nabla  _x b (v _1 )) - \alpha  \nabla  _y b (v _2 )) ) \cdot
		(\nabla  _x b (v _1 ) -\nabla  _y b (v _2 ))  dy ds dx dt \leq 0 ,
\end{align*}
where
\begin{equation*}
Q _1 = \{  (y,s ) \in Q ;~~ b(v_1 (y,s)) \in E \} ,
~~~
Q _2 = \{  (x,t ) \in Q ;~~ b(v_2 (x,t)) \in E \} 
\end{equation*}
and $E$ is defined as \eqref{null-set}.
Moreover,
let $Z _2= Z_2 ( y ,s  ,x,t )$ 
be a non-negative smooth function such that
\begin{align*}
(y ,s ) \mapsto Z _2 (y ,s  ,x,t  ) \in \mathscr{D} ([0, T ) \times \overline{\Omega} ) 
		~~~~&\text{ for each fixed } (x ,t ) \in Q ,\\
(x ,t ) \mapsto Z _2 (y ,s  ,x,t ) \in \mathscr{D} ([0, T ) \times \Omega ) 
		~~~~&\text{ for each fixed } (y ,s ) \in Q .		
 \end{align*}
Then we get
\begin{align*}
&\int _{Q \times Q} H ^{\circ } (v ^+ _ 1  - v ^+ _2 )
	\LD  
 -(g (v ^+ _1 )  -g (v ^+ _2 ) ) (\partial _t Z _1+ \partial _s Z_1) 	
	\right . \\
&\hspace{20mm} 
+(\alpha (\nabla  _x b (v ^+ _1 )) - \alpha (\nabla  _y b (v ^+ _2 )) ) 
		\cdot (\nabla _x Z _1+ \nabla _y Z_1 ) 
\\
&\hspace{40mm} 
\left .	- (f _1 - (1- H^{\circ } (v ^- _2 )) f_2  ) H^{\circ } (v ^+ _1 )) Z_1 \RD dyds dxdt \\
&\hspace{15mm} 
 -\int_{Q} \int _{\Omega }
	(\xi ^+_{01} (y) -g (v ^+_2 (x ,t )) ) _+  Z _1(y, 0 , x ,t  ) dy dx dt \\
&\hspace{15mm}
 -\int_{Q} \int _{\Omega }
	(g (v ^+ _1 (y ,s)) - \xi ^+_{02} (x)  ) _+   Z _1(y ,s, x , 0 ) dy dx ds   \\
& \leq 
- \lim_{\lambda \to 0} 
\int_{ (Q \setminus Q _1 ) \times (Q \setminus Q _2 ) }  
 		H ' _{\lambda } ( b (v ^+_1 ) - b (v ^+_2 )) Z_1 \\
&\hspace{20mm} \times 
(\alpha (\nabla  _x b (v ^+_1 )) - \alpha (\nabla  _y b (v ^+_2 ) ) \cdot
		(\nabla  _x b (v ^+_1 ) -\nabla  _y b (v ^+_2 ))  dy ds dx dt  \\
&\leq 0 ,
\end{align*}
where $w ^+ $ and $w ^- $ denote the positive part and negative part of $w$.

Therefore we can follow the argument via the convergence of mollification
given in \S 5 of Carrillo \cite{C}
and obtain 
\begin{equation}
\begin{split}
&\int _{Q} H^{\circ } (v_ 1 - v_ 2) \LD - (g (v _1 ) - g(v_2 ) ) \partial _t \zeta \right . \\
&\hspace{15mm} \left . 
+(\alpha (\nabla b (v_1 )) - \alpha (\nabla b (v_2 )) ) \cdot \nabla \zeta
 - (f _1 - f_2 ) \zeta  \RD dx dt  \\
&\hspace{30mm}
- \int _{\Omega } (\xi _{01} - \xi _{02} ) _+ \zeta (x , 0 ) dx \leq 0 
\end{split}
\label{E-Cor-01} 
\end{equation}
for every $\zeta \in \mathscr{D} ([ 0 , T ) \times \overline{\Omega } )$.
By putting $\zeta (x,t ) = \zeta (t)$ in \eqref{E-Cor-01},
we derive \eqref{comperi} (see Corollary 10 of \cite{C}).
\end{proof}
\noindent Especially, we can show that
\begin{Co}
Assume (H.$\alpha $), (H.$\beta $),  \eqref{Unique}, and $\beta $ is single-valued.
Then the solution to (P) given in Theorem \ref{MTh-01} is unique.
\end{Co}



\end{document}